\documentclass[12pt]{article}

\usepackage{amsthm}

\usepackage[left=2cm,right=2cm,top=2cm,bottom=2cm,bindingoffset=0cm]{geometry}

\usepackage[T2A]{fontenc}
\usepackage[utf8]{inputenc}
\usepackage[english]{babel}
\usepackage{amsmath}
\usepackage{amsfonts}
\usepackage{mathtools}
\usepackage{graphicx}
\usepackage{systeme}
\usepackage{tikz}
\usepackage{blindtext}
\usepackage{amssymb}
\usepackage{hyperref}
\usepackage{bbm}
\usepackage{cite}
\usepackage[normalem]{ulem}
\graphicspath{ {./} }

\newtheorem{theorem}{Theorem}[section]
\newtheorem{corollary}{Corollary}[theorem]
\newtheorem{lemma}{Lemma}[section]
\newtheorem*{lemma*}{Lemma}
\newtheorem*{den*}{Denotion}

\newtheorem*{remark*}{Remark}
\newtheorem{construction}{Construction}

\newtheorem{conjecture}{Conjecture}

\theoremstyle{remark}
\newtheorem*{remark}{Remark}

\theoremstyle{definition}
\newtheorem{definition}{Definition}[section]

\title{Spectrum of Johnson graphs\footnote{This work is done under the financial support of the grant of the president of the Russian Federation no. Nsh-775.2022.1.1 and the grant of Theoretical Physics and Mathematics Advancement Foundation “BASIS”.}}
\author{Mikhail Koshelev\footnote{Moscow State University, Moscow Institute of Physics and Technology, mkoshelev99@gmail.com}}
\date{}

\begin{document}

\maketitle

\begin{abstract}
    In this paper we prove new bounds on the second eigenvalue of Johnson graphs. We then apply these bounds to obtain new results on the modularity of Johnson graphs and their random subgraphs, hamiltonicity of Johnson graphs and thresholds of the appearance of the hamilton cycles and giant components. Futhermore, we provide general bounds on vertex connectivity and the stability of the modularity for $(n, d, \lambda)$-graphs.
    
\end{abstract}

\section{Introduction}

In this paper we mainly work with Johnson graphs, also known as $G(n, r, s)$ graphs. The vertices of such graphs are $n$-dimensional vectors $v \in \{0,1\}^n$ such that $\|v\|^2 = r$. An edge is drawn between two vertices if and only if the inner product of the corresponding vectors is equal to $s$.

Johnson graphs have found their applications in various topics of combinatorial geometry, Ramsey theory, coding theory and other branches of modern combinatorics. There exists a vast literature on such graphs, see, for example,~\cite{Codes, KK, FrW, LS, RaiMeta}.

The property we consider in this paper is the spectrum of such graphs. It has received a lot of attention in the past 50 years. It appears that one can say a lot about a graph just by looking at its eigenvalues. Before we continue, let us give some definitions.

\begin{definition}
    Let $G$ be a $d$-regular graph on $n$ vertices. Let $A$ be the adjacency matrix of $G$. We then call the sequence $\lambda_1 \geq \lambda_2 \dotsc \geq \lambda_n$ of the eigenvalues of $A$ the spectrum of $G$. We denote the degree of a vertex in a $d$-regular graph $G$ as $d(G)$.
\end{definition}

\begin{definition}
Let $G$ be a $d$-regular graph on $n$ vertices. Then 
    $$
        \lambda(G) := \max\{\lambda_2, -\lambda_n\} = \max\{|\lambda_2|, \dotsc, |\lambda_n|\}.
    $$
\end{definition}

The following theorem about the spectrum of $G(n, r, s)$ graphs is known (see \cite{Spectrum}, theorem 4.6).

\begin{theorem}[\cite{Spectrum}]
\label{Eigenvalues}
The set of eigenvalues of $G(n, r, s)$ coincides with the set $\{E_{r - s}(i), 0 \leq i \leq r\}$, where 
\begin{multline*}
    E_k(i) = \sum\limits_{j=0}^{k}{(-1)^{k-j}\binom{r-j}{k-j}\binom{r-i}{j}\binom{n-r+j-i}{j}} = \sum\limits_{j = 0}^{k}{(-1)^{j}\binom{i}{j}\binom{r - i}{k - j}\binom{n - r - i}{k - j}}.
\end{multline*}
Moreover, the multiplicity of $ E_{r - s}(i)$ is equal to $\binom{n}{i} - \binom{n}{i - 1}$. Note that in these formulas some of the binomial coefficients $\binom{n}{k}$ might have either $k < 0$ or $k > n$. In this case we assume them to be equal to zero.
\end{theorem}

Unfortunately, little is known about the behaviour of $\lambda(G(n, r, s))$. We are only aware of two results in this area. The first one is due to Lov\'asz and covers the case $s = 0$.

\begin{theorem}[\cite{LovLambda}]
\label{LovT}
$\lambda(G(n, r, 0)) = \frac{r}{n - r}d(G(n, r, 0))$.
\end{theorem}

The second one was introduced in \cite{Brouwer} and works for a wider range of parameters.

\begin{theorem}[\cite{Brouwer}]
\label{BrT}
Let $(r - s)(n - 1) \geq r(n - r)$. Then $$\lambda(G(n, r, s)) = |E_{r - s}(1)| = \frac{|sn - r^2|}{r(n - r)}d(G(n, r, s)).$$
\end{theorem}

Note that the condition $(r - s)(n - 1) \geq r(n - r)$ from the last theorem is equivalent to $r^2 \geq sn + r - s$. This means that nothing is known when, say, $r^2$ is smaller than $sn$. This is unfortunate since we can use the spectrum of such graphs to obtain progress in various problems, some of them are discussed later in this paper.

In this paper we are going to present a bunch of new results about the spectrum of Johnson graphs and then apply them to obtain new bounds for various characteristics of Johnson graphs. We also introduce some theorems that work for arbitrary $(n, d, \lambda)$-graphs. Here we use the notation $(n, d, \lambda)$-graph for a $d$-regular graph $G$ on $n$ vertices such that $\lambda(G) = \lambda$.

The rest of the paper is organized as follows. In Section 2 we obtain new bounds for $\lambda(G(n, r, s))$ for different values of parameters. In Section 3 we show some applications of the results of Section 2 and prove new results on the modularity, hamiltonicity and the threshold of the appearance of the Hamilton cycle and the giant component in such graphs.

\section{Spectrum of $G(n, r, s)$}

\subsection{Formulation of the results}

We are going to prove two theorems. The first one works for the symmetric case (i.e., $n=4s$, $r=2s$). These graphs are of great importance in combinatorial geometry, where they were used to disprove Borsuk's conjecture and to obtain first nontrivial bounds on Borsuk's number (see \cite{KK, RaiMeta}).

\begin{theorem}
\label{Th2}
There exists an integer $N$ such that for all $s > N$
$$
\lambda(G(4s, 2s, s)) = |E_{s}(2)| = \frac{4s - 2}{s^2}\binom{2s - 2}{s - 1}^2 = \frac{1}{4s - 2}E_s(0) = O\left(\frac{d(G(4s, 2s, s))}{s}\right).
$$
\end{theorem}

The second one covers the cases $n \gg r \gg s$ as well as constant $r$ and $s$. In the case of constant $r$ and $s$ we obtain the exact value of $\lambda(G(n, r, s))$ for all big enough $n$. In the case of $n \gg r \gg s$ the situation is a bit more complex. Under the additional condition $r = O(\sqrt{n})$ we obtain the following upper bound on the asymptotics of $\lambda$:
$$
\lambda(G(n, r, s)) = O\left(\frac{s}{r}d(G(n, r, s))\right).
$$
If, furthermore, either $s$ tends to infinity or $r$ is $o(\sqrt{n})$ we can derive the exact asymptotics of $\lambda$:
$$
\lambda(G(n, r, s)) = (1 + o(1))\frac{s}{r}d(G(n, r, s)).
$$
Finally, if none of the above is satisfied, we can still obtain a bound on the asymptotics of $\lambda$:
$$
\lambda(G(n, r, s)) = o\left(d(G(n, r, s))\right).
$$

\begin{theorem}
\label{ThMerge}
\begin{enumerate}
    \item Let $r > s \geq 1$. Then there exists an integer $N$ such that $\lambda(G(n, r, s)) = E_{r - s}(1) \sim \frac{s}{r}d(G(n, r, s))$ for all $n > N$.
    \item Let $n \gg r(n) \gg s(n)$, $r^2(n) \gg n$. Then $\lambda(G(n, r, s)) = o(d(G(n, r, s))).$
    \item Let $r(n) = O(\sqrt{n}), r \gg s \geq 1$. Then $\lambda(G(n, r, s)) = O\left(\frac{s}{r}d(G(n, r, s))\right).$
    Moreover, if at least one of  
    \begin{enumerate}
    \item $s = \omega(1)$;
    \item $r = o(\sqrt{n})$,
    \end{enumerate}
    holds, then $\lambda(G(n, r, s)) = (1 + o(1))\frac{s}{r}d(G(n, r, s)).$
\end{enumerate}
\end{theorem}

\begin{remark}
In this text we use the denotation $f(n) \gg (\ll) g(n)$ for functions $f$ and $g$ such that $\frac{f(n)}{g(n)} \to \infty (\to 0)$ as $n \to \infty$.
\end{remark}

\begin{remark}
It is also possible to prove that for any sequences $r(n), s(n)$ s.t. $n \gg r \gg s$ we have $\lambda(G(n, r, s)) = o(d(G(n, r, s)))$. This proof is however quite technical and thus we will not give it.
\end{remark}

Before we proceed with the proof of these theorems, we want to mention that for all triples $(n, r, s)$ with $2r - s \leq n \leq 150$ we have calculated $\lambda(G(n, r, s))$ with the help of the computer. In all these cases $\lambda(G(n, r, s))$ was equal to either $E_{r - s}(1)$ or $E_{r - s}(2)$. We thus believe that the following conjecture holds.

\begin{conjecture}
    $\lambda(G(n, r, s)) = \max\{|E_{r - s}(1)|, |E_{r - s}(2)|\}$.
\end{conjecture}

\subsection{Proofs}

Before we start with the proofs of mentioned results, let us state a useful lemma.

\begin{lemma}
\label{exp2}
Let $C_{n, r, s}(i, j) = \binom{i}{j}\binom{r - i}{r - s - j}\binom{n - r - i}{r - s - j}$. Then 
\begin{multline*}
    C_{n, r, s}(i, j) = \frac{(r-s-j+1)^2}{(r - i + 1)(n - r - i + 1)}C_{n, r, s}(i - 1, j - 1) +\\+ \frac{(s - i + j + 1)(n - 2r + s + j - i + 1)}{(r - i + 1)(n - r - i + 1)}C_{n, r, s}(i - 1, j).
\end{multline*}
\end{lemma}

The proof is quite straightforward. We consider three cases, $j = 0, j = i$ and $1 \leq j \leq i - 1$ and for each of the cases we check the equality by direct calculations. For a rigorous proof of this lemma see Appendix 1.

\subsubsection{Proof of Theorem \ref{Th2}}

For the sake of convenience we will once again give the statement of theorems in the beginning of the proof.

{\bf Theorem. }{\it There exists an integer $N$ such that for all $s > N$
$$
\lambda(G(4s, 2s, s)) = |E_{s}(2)| = \frac{4s - 2}{s^2}\binom{2s - 2}{s - 1}^2 = \frac{1}{4s - 2}E_s(0) = O\left(\frac{d(G(4s, 2s, s))}{s}\right).
$$}

Let us use Lemma \ref{exp2} to obtain a more convenient expression for $E_{s}(i)$:
\begin{multline}
    \label{oth}
    E_{s}(i) = \sum\limits_{j = 0}^{i}{(-1)^{j}C_{4s, 2s, s}(i, j)} = \\ = \sum\limits_{j = 0}^{i}{(-1)^{j}\left(\frac{(s - j + 1)^2}{(2s - i + 1)^2}C_{4s, 2s, s}(i - 1, j - 1) + \frac{(s - i + j + 1)^2}{(2s - i + 1)^2}C_{4s, 2s, s}(i - 1, j)\right)} = \\ = \sum\limits_{j = 0}^{i - 1}{(-1)^{j}\left(\frac{(s - i + j + 1)^2}{(2s - i + 1)^2} - \frac{(s - j)^2}{(2s - i + 1)^2}\right)C_{4s, 2s, s}(i - 1, j)} = \\ = \sum\limits_{j = 0}^{i - 1}{(-1)^{j}\frac{2j - i + 1}{2s - i + 1}C_{4s, 2s, s}(i - 1, j)}.
\end{multline}

Let us now consider two cases.

\begin{enumerate}
    \item $i > 4\log_{2}s$. Then 
    \begin{multline*}
        |E_{s}(i)| = \left|\sum\limits_{j = 0}^{s}{(-1)^{j}\binom{i}{j}\binom{2s - i}{s - j}^2}\right| \leq \sum\limits_{j = 0}^{s}{\binom{i}{j}\binom{2s - i}{s - j}^2} \leq \\ \leq (s + 1)2^{i}2^{4s-2i} = (s + 1)2^{4s - i} \leq (s + 1)2^{4s - 4\log_{2}s} = \frac{(s + 1)2^{4s}}{s^4} = o\left(\frac{E_{s}(0)}{s}\right) = o(E_{s}(2)).
    \end{multline*}

    \item $ i \leq 4\log_{2}s$. Let us first prove by induction that for such $i$, $0 \leq j \leq i$ and big enough $s$ the inequality $C_{4s, 2s, s}(i, j) \leq \left(\frac{101}{200}\right)^iC_{4s, 2s, s}(0, 0)$ holds. For $i = 0$ this is obvious. To prove the induction step it is sufficient to use Lemma \ref{exp2}. 
    \begin{multline*}
        C_{4s, 2s, s}(i, j) = \frac{(s - j + 1)^2}{(2s - i + 1)^2}C_{4s, 2s, s}(i - 1, j - 1) + \frac{(s - i + j + 1)^2}{(2s - i + 1)^2}C_{4s, 2s, s}(i - 1, j) \leq \\ \leq \left(\frac{(s + 1)^2}{(2s - 4\log_{2}s)^2} + \frac{(s + 1)^2}{(2s - 4\log_{2}s)^2}\right)\left(\frac{101}{200}\right)^{i-1}C_{4s, 2s, s}(0, 0) = \\ = \frac{(s + 1)^2}{2(s - 2\log_{2}s)^2}\left(\frac{101}{200}\right)^{i-1}C_{4s, 2s, s}(0, 0) \leq \left(\frac{101}{200}\right)^iC_{4s, 2s, s}(0, 0).
    \end{multline*}
    So we have proved that $C_{4s, 2s, s}(i, j) \leq \left(\frac{101}{200}\right)^iC_{4s, 2s, s}(0, 0)$. Let us now use (\ref{oth}) to obtain the bound on $E_{s}(i)$.
    \begin{multline*}
        |E_{s}(i)| = \left|\sum\limits_{j = 0}^{i - 1}{(-1)^{j}\frac{2j - i + 1}{2s - i + 1}C_{4s,2s,s}(i - 1, j)}\right| \leq \sum\limits_{j = 0}^{i - 1}{\frac{|2j - i + 1|}{2s - i + 1}C_{4s,2s,s}(i - 1, j)} \leq  \\ \leq \frac{i(i + 1)}{2s - i + 1}\left(\frac{101}{200}\right)^{i-1}C_{4s,2s,s}(0, 0) \leq \frac{K_i}{2s - 4\log_{2}s + 1}E_{s}(0),
    \end{multline*}
    where $K_i = i(i + 1)\left(\frac{101}{200}\right)^{i - 1}$. One can easily see that for $i \geq 9$ the value of $K_i$ is strictly smaller than $\frac{1}{2}$ and thus for all big enough $s$ we have $\forall i \geq 9: |E_s(i)| < |E_s(2)|$. The only remaining thing to consider is the case $i < 10$. Let us first notice that for odd $i$ we have $E_{s}(i) = 0$, as $C_{4s, 2s, s}(i, j) = C_{4s, 2s, s}(i, i - j), (-1)^j = -(-1)^{i - j}$. So the last thing to prove is that for big enough $s$
    $$
    \min\{|E_{s}(2)|, |E_{s}(4)|, |E_{s}(6)|, |E_{s}(8)|\} = |E_{s}(2)| = \frac{4s-2}{s^2}\binom{2s-2}{s-1}^2 = \frac1{4s-2}E_{s}(0).
    $$
    
    This can be done by direct calculation, so we omit it here. For additional details about the exact values of $E_s(2k), k = 2, 3, 4$, see Appendix 2.
\end{enumerate}

\subsubsection{Proof of Theorem \ref{ThMerge}}

{\bf Theorem. } {\it
\begin{enumerate}
    \item Let $r > s \geq 1$. Then there exists an integer $N$ such that $\lambda(G(n, r, s)) = E_{r - s}(1) \sim \frac{s}{r}d(G(n, r, s))$ for all $n > N$.
    \item Let $n \gg r(n) \gg s(n)$, $r^2(n) \gg n$. Then $\lambda(G(n, r, s)) = o(d(G(n, r, s))).$
    \item Let $r(n) = O(\sqrt{n}), r \gg s \geq 1$. Then $\lambda(G(n, r, s)) = O\left(\frac{s}{r}d(G(n, r, s))\right).$
    Moreover, if at least one of  
    \begin{enumerate}
    \item $s = \omega(1)$;
    \item $r = o(\sqrt{n})$,
    \end{enumerate}
    holds, then $\lambda(G(n, r, s)) = (1 + o(1))\frac{s}{r}d(G(n, r, s)).$
\end{enumerate}
}

We start with the case of constant $r$ and $s$. Let us calculate the asymptotics of $E_{r - s}(i)$. For $i \leq s$ we have
\begin{multline*}
    E_{r - s}(i) = \binom{r - i}{r - s}\binom{n - r - i}{r - s} + \sum_{j = 1}^{r - s}(-1)^j\binom{i}{j}\binom{r - i}{r - s - j}\binom{n - r - i}{r - s - j} = \\ = \binom{r - i}{r - s}\binom{n - r - i}{r - s} + \sum_{j = 1}^{r - s}C_jn^{r - s - j} = \\ = \binom{r - i}{r - s}\binom{n - r - i}{r - s} + o(n^{r - s}) \sim \frac{1}{(r - s)!}\binom{r - i}{r - s}n^{r - s}.
\end{multline*}
For $i > s$ the asymptotic equivalence
$$
\binom{r - i}{r - s}\binom{n - r - i}{r - s} + o(n^{r - s}) \sim \frac{1}{(r - s)!}\binom{r - i}{r - s}n^{r - s})
$$
does not hold, as the coefficient before $n^{r-s}$ is equal to 0. Nevertheless for such $i$ the following equality holds $E_{r - s}(i) = o(n^{r - s})$.
As one can see, for big enough $n$ the biggest absolute value of $E_{r - s}(i)$ is $|E_{r - s}(1)|   \sim \frac{1}{(r - s)!}\binom{r - 1}{r - s}n^{r - s} = \frac{s}{r}\binom{r}{s}\frac{n^{r-s}}{(r-s)!} \sim \frac{s}{r}\binom{r}{s}\binom{n - r}{r-s} = \frac{s}{r}d(G(n, r, s))$.

For the other two cases we need to make use of Lemma \ref{exp2}. First of all, just like in Theorem \ref{Th2} we want to obtain a bound on $C_{n, r, s}(i, j)$. Since in our case we have $n \gg r \gg s$, we may assume that $n > 100r, r > 100s$. Then the following bound is pretty straightforward: $$\forall i \leq \frac{r}{10}: C_{n, r, s}(i, j) \leq \left(\frac{1}{2}\right)^iC_{n, r, s}(0, 0).$$ 
For a rigorous proof of this inequality see Appendix 3.

The second thing we want to obtain is a new formula for $E_{r - s}(i)$. By applying Lemma \ref{exp2} to the formula $E_{r - s}(i) = \sum\limits_{j = \max\{0, i - s\}}^{i}{(-1)^{j}C_{n, r, s}(i, j)}$ we obtain the following equality:
\begin{multline}
    \label{oth2}
    E_{r - s}(i) = \sum\limits_{j = \max\{0, i - s\} + 1}^{i} (-1)^{j} \left(\frac{(r - s - j + 1)^2}{(r - i + 1)(n - r - i + 1)} - \frac{(s - i + j)(n - 2r + s + j - i)}{(r - i + 1)(n - r - i + 1)}\right)\times \\ \times C_{n, r, s}(i - 1, j - 1) +  Q(n,r,s,i),
\end{multline}
where $Q(n, r, s, i)$ can be shown to be $O\left(\frac{r}{n}C_{n, r, s}(0, 0)\right)$. For the proof and the exact value of $Q(n, r, s, i)$ see Appendix 4.

We are now ready to apply the formulas above to prove the desired results. By combining the bound on $C_{n, r, s}(i, j)$ and formula (\ref{oth2}), we get that for all $i \leq \frac{r}{10}$
\begin{multline}
\label{SmallI}
|E_{r - s}(i)| \leq \frac{i}{2^{i - 1}}C_{n, r, s}(0, 0)\times \\ \times\max_{j \in \{0, \dotsc, i\}}\left|\frac{(r - s - j + 1)^2 - (s - i + j)(n - 2r + s + j - i)}{(r - i + 1)(n - r - i + 1)}\right|  + O\left(\frac{r}{n}C_{n, r, s}(0, 0)\right).
\end{multline}

It is easy to see that the numerator of $\frac{(r - s - j + 1)^2 - (s - i + j)(n - 2r + s + j - i)}{(r - i + 1)(n - r - i + 1)}$ is a linear function of $j$ while the denominator does not depend on $j$ at all. This means that the maximum value of this function is attained either at $j = 0$ or $j = i$.

From now on we consider the two cases of the theorem separately. 

For the case $r^2 \gg n$ we can use the expression (\ref{SmallI}) to prove that for $1 \leq i \leq \sqrt{n} \leq r/10$ we have $E_{r - s}(i) = o(C_{n, r, s}(0, 0))$. To do so we just need to prove that for such $i$
$$
\left|\frac{(r - s + 1)^2 - (s - i)(n - 2r + s - i)}{(r - i + 1)(n - r - i + 1)}\right| = o(1),
$$
$$
\left|\frac{(r - s - i + 1)^2 - s(n - 2r + s)}{(r - i + 1)(n - r - i + 1)}\right| = o(1),
$$
which is quite easy to verify by direct calculations.

Now it remains to consider $i > \sqrt{n}$. 
Let us use the formula
$$
    E_{r - s}(i) = \sum\limits_{j = 0}^{r - s}{(-1)^{j}\binom{i}{j}\binom{r - i}{r - s - j}\binom{n - r - i}{r - s - j}}.
$$
We have
$$
|E_{r - s}(i)| \leq \max_{j \in \{0, \dotsc, r - s\}}\binom{n - r - i}{r - s - j}\sum\limits_{j = 0}^{r - s}{\binom{i}{j}\binom{r - i}{r - s - j}} = \binom{r}{r - s}\max_{j \in \{0, \dotsc, r - s\}}\binom{n - r - i}{r - s - j}.
$$
As $n - 2r > 2(r - s)$ we obtain that the maximum value of the right hand side is attained at $j = 0$. We thus have the following bound on $|E_{r - s}(i)|$: 
$$
|E_{r - s}(i)| \leq \binom{r}{r - s}\binom{n - r - i}{r - s}.
$$
It remains to notice that $\binom{r}{r - s}\binom{n - r - i}{r - s} = o\left(\binom{r}{s}\binom{n - r}{r - s}\right)$. 
    
The proof for the third case of the theorem starts in a similar way. We can use the expression (\ref{SmallI}) to prove that for $1 \leq i \leq 2s \leq r/10$ we have $E_{r - s}(i) = O\left(\frac{s}{r}C_{n, r, s}(0, 0)\right)$. To do so we just need to prove that for such $i$
$$
\left|\frac{(r - s + 1)^2 - (s - i)(n - 2r + s - i)}{(r - i + 1)(n - r - i + 1)}\right| = O\left(\frac{s}{r}\right),
$$
$$
\left|\frac{(r - s - i + 1)^2 - s(n - 2r + s)}{(r - i + 1)(n - r - i + 1)}\right| = O\left(\frac{s}{r}\right).
$$
Furthermore, if at least one of the additional assertions is true, these bounds can be strengthened to obtain the following inequalities
$$
\left|\frac{(r - s + 1)^2 - (s - i)(n - 2r + s - i)}{(r - i + 1)(n - r - i + 1)}\right| \leq (1 + o(1))\frac{s}{r},
$$
$$
\left|\frac{(r - s - i + 1)^2 - s(n - 2r + s)}{(r - i + 1)(n - r - i + 1)}\right| \leq (1 + o(1))\frac{s}{r}.
$$
The proofs here are straightforward, so we omit them.

The last part ($i > 2s$) here is however a bit different. To prove it, let us apply the following idea from \cite{Brouwer}. Consider the adjacency matrix $A$ of $G(n,r,s)$. Then every entry on the diagonal of $A^2$ is equal to $d = d(G(n,r,s))$ and thus $$d\binom{n}{r} = \text{tr}(A^2) = \sum_{i = 0}^{\binom{n}{r}}\lambda_i^2 = \sum_{i = 0}^{r}\left(\binom{n}{i} - \binom{n}{i - 1}\right)E_{r-s}(i)^2.$$ From this equality the following bound on $|E_{r - s}(i)|$ is easily obtained $$|E_{r - s}(i)| \leq \sqrt{\frac{\binom{n}{r}\binom{r}{s}\binom{n - r}{r - s}}{\binom{n}{i} - \binom{n}{i - 1}}} = \binom{r}{s}\binom{n - r}{r - s}\sqrt{\frac{\binom{n}{r}}{(\binom{n}{i} - \binom{n}{i - 1})\binom{r}{s}\binom{n - r}{r - s}}}.$$ 
So if suffices to prove that $\frac{\binom{n}{r}}{\left(\binom{n}{i} - \binom{n}{i - 1}\right)\binom{r}{s}\binom{n - r}{r - s}} = O(1/n)$ for $i > 2s$. While the proof of this fact is not very hard it is pretty technical and thus we move its proof to Appendix 5.

Finally, if at least one of the additional assertions is true, we have $$E_{r - s}(1) = \frac{sn - r^2}{r(n - r)}d(G(n,r,s)) = (1 - o(1))\frac{s}{r}d(G(n,r,s)),$$ which completes the proof.

\section{Application of the results of Section 2}

In this part of the paper we apply the results of the previous section to obtain new results about $G(n, r, s)$ graphs.

The first property that we consider is the modularity of these graphs. Informally speaking, modularity measures how well the vertices of the graph can be split into disjoint sets with most of the edges connecting two vertices from the same part. Modularity was introduced in \cite{Item22} by Newman and Girvan and recieved a lot of attention due to its applications in different fields of studies, including biology, physics, sociology and computer science.

A lot of recent works on modularity are devoted to its calculation for different families of graphs, including stars, hypercubes (see \cite{Stars}), almost complete graphs (see \cite{AlmostFull}), trees with small maximum degree (see \cite{Trees}) and various families of random graphs. For example, some bounds on the modularity were obtained for preferential attachment models (see \cite{Item25}), random $d$-regular graphs (see \cite{Item5}) and the classical Erd\H os--R\'enyi model (see \cite{McD}).

In Section 3.1 we apply our knowledge about the spectrum of Johnson graphs to obtain new bounds on their modularity. 

The other thing we consider here is the stability of the modularity. First of all let us introduce a definition of a random subgraph of a graph $G$.

\begin{definition}
Let $G$ be a graph. Then $G(p)$ is a random graph obtained from $G$ by removing every edge with probability $1 - p$ independently from all other edges.
\end{definition}

\begin{remark}
For $G(n, r, s)$ graphs writing $G(n,r,s)(p)$ is not very convenient and thus we use the notation $G_p(n,r,s)$ instead.
\end{remark}

In Section 3.2 we  prove a general theorem about the stability of the modularity for $(n, d, \lambda)$-graphs, or, in other words, that for big enough $p$ graphs $G$ and $G(p)$ almost surely have the same modularity. As a simple corollary we obtain theorems about the modularity of $G_p(n, r, s)$ graphs.

Another interesting property that we consider in Section 3.3 is the emergence of a giant component in $G_p(n,r,s)$ graphs. In \cite{ErdGC} the following result about the Erd\H os--R\'enyi random graph evolution was proven. It turned out that for $p < \frac{1-\varepsilon}{n}$ the graph $K_n(p)$ with high probability (w.h.p.) has only small components, while for $p > \frac{1+\varepsilon}{n}$ there w.h.p. exists a component of size $\Omega(n)$. Here we use the term $``$with high probability$"$ for such a sequence of events $\{Q_n\}_{n = 1}^{\infty}$ that $\mathbf{P}(Q_n) \to 1, n \to \infty$. It is interesting to generalize this result and understand what are the conditions for graph $G$ to evolve in the same way. One of the known sufficient conditions is formulated in terms of $\lambda(G)$, so we apply it to find a threshold probability for the property of containing a giant component for $G(n, r, s)$ graphs. 

The last property we consider is the hamiltonicity of $G(n, r, s)$ and $G_p(n, r, s)$ graphs. In \cite{Hamil} the following conjecture was stated.

\begin{conjecture}
    For any triple $(n, r, s) \notin \{(5, 2, 0), (5, 3, 1)\}$ the graph $G(n, r, s)$ contains a Hamiltonian cycle if and only if $G(n, r, s)$ is connected.
\end{conjecture}

There are a lot of partial results on this conjecture (see \cite{Hamil, H1, H2}). In Section 3.4 we obtain another partial result using our knowledge about the spectrum of $G(n, r, s)$. Moreover, we obtain a new bound on the vertex connectivity of $(n, d, \lambda)$-graphs which is a direct generalization of the result from \cite{Kriv06}.

Finally in Section 3.5 we consider the threshold probability of the appearance of a Hamilton cycle in $G_p(n, r, s)$. This threshold is known for the classical Erd\H os-R\'enyi model and it is of great interest to obtain similar results for other $G(p)$ graphs. We apply our result about the spectrum of $G(4s, 2s, s)$ to obtain the threshold for such graphs.

\subsection{Modularity of Johnson graphs}

Before we start, let us introduce some notions. 

\begin{itemize}
    \item $e(G)$ is the number of edges in $G$.
    \item $e(V)$ is the number of edges of $G$ with both endpoints in $V$ 
    \item $e(V, p)$ --- same as the previous one, but for $G(p)$.
    \item $e(V, U)$ is the number of edges of $G$ with one endpoint in $V$ and the other in $U$.
    \item $e(V, U, p)$ --- same as the previous one, but for $G(p)$.
    \item $\deg(v, p)$ is the degree of vertex $v$ in $G(p)$.
    \item $V_G$ --- the set of all vertices of a graph $G$.
    \item $E_G$ --- the set of all edges of a graph $G$.
    \item $\overline{V} \coloneqq V_G \setminus V$.
\end{itemize}

Let us now define modularity rigorously. 

\begin{definition} 
Let $\mathcal{A} = \{A_1, A_2, ..., A_k\}$ be a partition of the vertices of a graph $G$. The \textit{edge contribution} is then defined as $\sum\limits_{i = 1}^{k}{\frac{e(A_k)}{e(G)}}$, where $e(A) = |\{(v_1,v_2) \in E_G | v_1,v_2 \in A\}|$. 
\end{definition}

\begin{definition}
Let $\mathcal{A} = \{A_1, A_2, ..., A_k\}$ be a partition of the vertices of a graph $G$. The value $$\sum\limits_{i = 1}^{k}{\frac{(\sum_{v\in A_k}{\deg(v)})^2}{4e^2(G)}}$$ is called the \textit{degree tax}.
\end{definition}

\begin{remark}
Note that for $d$-regular graphs the degree tax can be rewritten as 
\begin{equation*}
    \sum\limits_{i = 1}^{k}{\frac{(|A_k|d)^2}{(d|V_G|)^2}} = \sum\limits_{i = 1}^{k}{\frac{|A_k|^2}{|V_G|^2}},
\end{equation*}
where $V_G$ is the set of vertices of $G$.
\end{remark}

\begin{definition}
Let $\mathcal{A} = \{A_1, A_2, ..., A_k\}$ be a partition of the vertices of a graph $G$. The \textit{modularity of partition} $\mathcal{A}$ is defined as the difference between the edge contribution and the degree tax of this partition:
\begin{equation*}
    q(\mathcal{A}) = \sum_{A \in \mathcal{A}}\frac{e(A)}{e(G)} -
    \sum_{A \in \mathcal{A}}\frac{(\sum_{v \in A}\deg(v))^2}{4e^2(V_G)}.
\end{equation*}
\end{definition}

\begin{definition}
The modularity of graph $G$ is defined as a maximum of modularities over all partitions of the vertices of $G$:
\begin{equation*}
    q^*(G) = \max_{\mathcal{A}}\{q(\mathcal{A})\}.
\end{equation*}
\end{definition}

In \cite{Rai} (see Theorem 4) the following bound on the modularity was obtained.

\begin{theorem}
\label{ThModEigen}
Let $\lambda_1 \geq \lambda_2 \geq \dotsc \geq \lambda_n$ be the eigenvalues of the adjacency matrix of a $d$-regular graph $G$ with $d \geq 3$. Then 
$$
q^{*}(G) \leq \frac{\lambda}{d},
$$
where $\lambda = \max\{\lambda_2, -\lambda_n\} = \max\{|\lambda_2|, \dotsc, |\lambda_n|\}$.
\end{theorem}

With this theorem we can easily obtain the following four corollaries bounding the modularity of $G(n,r,s)$ graphs from above.

\begin{corollary}[of Theorem \ref{LovT}]
\label{C3}
$$
q^{*}(G(n, r, 0)) \leq \frac{r}{n - r}.
$$
\end{corollary}

\begin{corollary}[of Theorem \ref{BrT}]
\label{C5}
Let $(r - s)(n - 1) \geq r(n - r)$. Then
$$
q^{*}(G(n, r, s)) \leq \frac{|sn - r^2|}{r(n - r)}.
$$
\end{corollary}

\begin{corollary}[of Theorem \ref{Th2}]
\label{C2}
$q^{*}(G(4s, 2s, s)) \leq \frac{1}{4s - 2} $ for all big enough $s$.
\end{corollary}

\begin{corollary}[of Theorem \ref{ThMerge}]
\label{C1}
Let $r > s \geq 1$. Then $\limsup\limits_{n \to \infty}{q^*(G(n, r, s))} \leq \frac{s}{r}$.
\end{corollary}

\begin{corollary}[of Theorem \ref{ThMerge}]
\label{C4}
Let $n \gg r(n) \gg s(n)$. Then 
$$
q^*(G(n, r, s)) \to 0, n \to \infty.
$$
\end{corollary}

Let us compare this to the previous results on the modularity. In \cite{Item11} the following general bound was introduced.

\begin{theorem}[\cite{Item11}]
\label{P1}
Let $r \geq 2$ and $1 \leq s \leq \left[\frac{r}2\right]$. Then
$$\displaystyle{\limsup_{n \to \infty}q^*(G(n,r,s))} \leq 1 - \frac{\binom{\left[\frac{r}2\right]}{s}}{2\binom{r}{s}}.$$
\end{theorem}

It was then improved for some cases in \cite{Der}.

\begin{theorem}[\cite{Der}]
\label{T2}
Let $\varepsilon \in (0, 1), s(n) \geq 1, \sqrt{n} \gg r(n) \geq -\frac1{\ln (1 - \varepsilon)}s^2 + 2s - 1$. Then $$\limsup_{n \to \infty}{q^*(G(n, r, s))} \leq f(\varepsilon),$$ where $f(\varepsilon) = \max\limits_{x \in [0, 1]}{\left(\frac{1 + x - x^2}{2 - x} - \max\left(\frac{x^2 - \varepsilon x}{1 - \varepsilon}, 0\right)\right)}$.
\end{theorem}

It is obvious that Corollaries \ref{C1} and \ref{C4} give us much stronger upper bounds on the modularity than both of these theorems.

It is also interesting to compare this new upper bound with the lower bound for constant $r$ and $s$. The best known lower bound was proved in \cite{Item17a}.

\begin{theorem}[\cite{Item17a}]
\label{P4}
Let $r > s \geq 1$. Then 
$$\liminf_{n \to \infty}{q^*(G(n, r, s))} \geq \frac{s}{2r - s}\left(1 + \left(\frac{r - s}{r}\right)^{\frac{2r}{s}}\right).$$
\end{theorem}

One can see that the new upper bound is no more than two times bigger when compared with the lower bound, as 
$$
\frac{s}{2r - s}\left(1 + \left(\frac{r - s}{r}\right)^{\frac{2r}{s}}\right) \geq \frac{s}{2r-s} \geq \frac{s}{2r}. 
$$

So we see that the spectral approach allows us to introduce good bounds on the modularity of $G(n,r,s)$ graphs. The only case when this approach does not give us anything new is the case of $r = 2, s = 1$. The modularity in this case was calculated precisely in \cite{Item9a}. We state this theorem here for the sake of completeness.

\begin{theorem}[\cite{Item9a}]
\label{P2}
\begin{equation*}
    q^*(G(n, 2, 1)) = \frac13 + \frac{2w(w-1)(w-2)}{3n(n-1)(n-2)} - \frac{w^2(w-1)^2}{n^2(n-1)^2} - \frac{4n - 2}{3n(n-1)} + \frac{w(w-1)(4w - 2)}{3n^2(n-1)^2}
\end{equation*}
for all $n\geq 5$, where $w = \lceil \frac{n}{2} \rceil + 1$. The limit of the right hand side as $n\to \infty$ equals $\frac{17}{48}$.
\end{theorem}

\subsection{Stability of modularity}

In this section we will use the denotation $q(\mathcal{A}, p)$ for the modularity of a partition $\mathcal{A}$ as a partition of $G(p)$.

First of all we prove that modularity of a $d$-regular graph is stable in some sense if $\frac{\lambda}{d} \leq c < 1$. 

\begin{theorem}
\label{ProbMain}
Let $S$ be an infinite subset of $\mathbb{N}$. Let $\{G_n\}_{n \in S}$ be a sequence of $(n, d(n), \lambda(n))$-graphs with $\limsup_{n \to \infty} \frac{\lambda(n)}{d(n)} < 1$ and let $p = p(n) \gg \frac1{\sqrt{d}}$. Consider a sequence of random graphs $\{G_{n}(p)\}$. Then
$$
\mathbf{P}\left(\limsup_{n \to \infty} |q^*(G_{n}) - q^{*}(G_{n}(p))| = 0\right) = 1.
$$
\end{theorem}

Note that by combining this theorem with Theorem \ref{ThModEigen} we can obtain the following bound on the modularity of $(n, d,\lambda)$-graphs.

\begin{theorem}
\label{Pr1}
Let $S$ be an infinite subset of $\mathbb{N}$. Let $\{G_n\}_{n \in S}$ be a sequence of $(n, d(n), \lambda(n))$-graphs with $\limsup\limits_{n \to \infty} \frac{\lambda(n)}{d(n)} \leq c$ and let $p = p(n) \gg \frac1{\sqrt{d}}$. Consider a sequence of random graphs $\{G_{n}(p)\}$. Then $$\mathbf{P}\left(\limsup\limits_{n \to \infty} q^{*}(G_n(p)) \leq c\right) = 1.$$
\end{theorem}

In the proof we use the following classical inequality proved in \cite{Hef}.

\begin{lemma}[Hoeffding]\label{Hf}
    \text{ } \\
    Let $X_1, \dotsc, X_m$ be independent random variables such that for each $i$ we have $\mathbf{P}(X_i \in [a_i, b_i]) = 1$ for some $a_i$ and $b_i$. Let $S_{m} = X_1 + \dotsc + X_m$. Then the following inequality holds.
    $$
    \mathbf{P}(|S_{m} - \mathbf{E}S_m| \geq \varepsilon m) < 2\exp\left(-\frac{2\varepsilon^2m^2}{\sum\limits_{i = 1}^{m}(b_i - a_i)^2}\right).
    $$
\end{lemma}

We also need a lemma bounding the number of edges between two parts of a graph. Here we use the following lemma for $(n, d, \lambda)$-graphs. For a proof, see \cite{Sp2} (Lemma 2.1).

\begin{lemma}
    \label{Eig}
    Let $G$ be an $(n, d, \lambda)$-graph. Then 
    $$
    e(S, \overline{S}) \geq \frac{(d - \lambda)|S|(n - |S|)}{n}.
    $$
\end{lemma}

We finally need an algorithm to modify partition without changing its modularity much.

\begin{construction}\label{algorithm}
    Let $G_n$ be a graph on $n$ vertices, and let $k \geq 2$, $k \in \mathbb{N}$. Let us also fix some positive integer $k$.
    We define an algorithm to construct two sets, $\mathcal{A}^{\prime}$ and $\mathcal{A}_{big}$ from a partition $\mathcal{A} = \{A_1, A_2, \ldots, A_m\}$. We also need an auxiliary set $A_{merged}$ which is initially empty.
    
    Iterate through $A_i, i\in\{1,2,\ldots,m\}$ :
    \begin{itemize}
        \renewcommand{\labelitemi}{\textbullet}
        \renewcommand{\labelitemii}{\textperiodcentered}
        \item If $|A_i| > \frac{n}{k}$, then add $A_i$ to $A^{\prime}$
        \item  If $|A_i| \leq \frac{n}{k}$ :
        \begin{itemize}
            \item $A_{merged} := A_{merged} \cup A_i$
            \item If $|A_{merged}| \geq \frac{n}{k}$, then $A_{merged}$ is added to $\mathcal{A}^{\prime}$ and $A_{merged} := \varnothing$
        \end{itemize}
    \end{itemize}
    
    After iterating through $A_i \in \mathcal{A}$ consider the set $A_{merged}$ :
    \begin{itemize}
        \renewcommand{\labelitemi}{\textbullet}
        \item If $|A_{merged}| > \frac{n}{k}$, then $A_{merged}$ is added to $\mathcal{A}^{\prime}$
        \item If $|A_{merged}| \leq \frac{n}{k}$, then $A_{merged}$ is ignored
    \end{itemize}
    
    Finally, $\mathcal{A}_{big}$ is defined in the following way:
    $$
        \mathcal{A}_{big} := \bigg\{A : A \in \mathcal{A}^{\prime}, |A| > \frac{2n}{k}\bigg\}.
    $$
\end{construction}

Note the following key property of this algorithm: any element from $\mathcal{A}_{big}$ is also present in $\mathcal{A}$ and any element from $\mathcal{A} \setminus \mathcal{A}_{big}$ has size of no more than $\frac{2n}{k}$.

\subsubsection{Auxiliary lemmas}

In this section we prove concentration inequalities on $e(G_n(p)), e(V, \overline{V}, p)$ and $\sum\limits_{v \in V}{\deg(v, p)}$, where $V$ is a subset of vertices of an $(n, d, \lambda)$-graph (with some restrictions on its size).

\begin{lemma}
    \label{LeG}
    Let $\{G_n\}$ satisfy the conditions of Theorem \ref{ProbMain} and let $p = p(n) = f(n)\frac{1}{\sqrt{d}}, f(n) \to \infty$. Then almost surely the following inequality holds for all sufficiently large $n$. 
    \begin{equation}
    \label{L1}
    \left|e(G_n(p)) - p\frac{dn}{2}\right| < \frac{p\log(f(n))}{f(n)}\frac{dn}{2}.
    \end{equation}
\end{lemma}

\begin{proof}
From Lemma \ref{Hf} for any fixed $n$ we have 
$$
\mathbf{P}\left(\left|e(G_n(p)) - p\frac{dn}{2}\right| \geq \frac{p\log(f(n))}{f(n)}\frac{dn}{2}\right) \leq 2\exp\left\{-\varepsilon^2dn\right\},
$$
where $\varepsilon = \frac{p\log(f(n))}{f(n)}$. Putting it all together we obtain
$$
\mathbf{P}\left(\left|e(G_n(p)) - p\frac{dn}{2}\right| \geq \frac{p\log(f(n))}{f(n)}\frac{dn}{2}\right) \leq 
$$
$$
\leq 2\exp\{-\log^2(f(n))n\}.
$$
It thus remains to prove that the probability of $\{G_n(p)\}$ to satisfy inequality (\ref{L1}) for all $n \geq N$ tends to 1 as $N$ tends to infinity.

$$
\mathbf{P}(G_n(p) \text{ satisfies inequality (\ref{L1}) for all } n \geq N) \geq \prod_{n \geq N}(1 - e^{-n\log^2(f(n))}).
$$
It remains to notice that the right hand side tends to 1 as $N \to \infty$.
\end{proof}

\begin{lemma}
    \label{LeA}
	    Let $\{G_n\}$ satisfy the conditions of Theorem \ref{ProbMain}, let $k$ be a positive integer and let $p = p(n) = f(n)\frac{1}{\sqrt{d}}, f(n) \to \infty$. Then almost surely the following inequality holds for all sufficiently large $n$.
	    $$\forall V \subset V_{G_n}, |V| \in \left[\frac{n}{k}, \frac{(k - 1)n}{k}\right] : |e(V, \overline{V}, p) - pe(V, \overline{V})| < \frac{p\log(f(n))}{f(n)}e(V, \overline{V}).$$
\end{lemma}

\begin{proof}
Let us fix some $n$ and bound the probability of the following event for some fixed  $V \subset V_G$: $|e(V, \overline{V}, p) - pe(V, \overline{V})| \geq \frac{p\log(f(n))}{f(n)}e(V, \overline{V})$. Taking $\varepsilon = \frac{p\log(f(n))}{f(n)}$ and applying Lemmas \ref{Hf} and \ref{Eig} we obtain 
$$
\mathbf{P}\left(|e(V, \overline{V}, p) - pe(V, \overline{V})| \geq \frac{p\log(f(n))}{f(n)}e(V, \overline{V})\right) < 2\exp\left\{-2\varepsilon^2e(V, \overline{V})\right\} \leq 
$$
$$
\leq 2\exp\left\{-2\frac{p^2\log^2(f(n))}{f^2(n)}\frac{(d-\lambda)|V||\overline{V}|}{|V_{G_n}|}\right\}.
$$
As we know that $d - \lambda \geq Cd$, we get that for all $V : |V| \in \left[\frac{n}{k}, \frac{(k - 1)n}{k}\right]$ the following bound holds.
$$
\mathbf{P}\left(|e(V, \overline{V}, p) - pe(V, \overline{V})| \geq \frac{p\log(f(n))}{f(n)}e(V, \overline{V})\right) < 2\exp\left\{-\frac{2}{d}\log^2(f(n))\frac{Cdn}{k^2}\right\},
$$
or, equivalently,
$$
\mathbf{P}\left(|e(V, \overline{V}, p) - pe(V, \overline{V})| \geq \frac{p\log(f(n))}{f(n)}e(V, \overline{V})\right) < 2\exp\left\{-C_0(k)\log^2(f(n))n\right\}.
$$
That means that the probability the property we want to prove does not hold is no more than
$$
2^n\exp\left\{-C_0n\log^2(f(n))\right\} \leq \exp\left\{n\ln 2 - C_0n\log^2(f(n))\right\} < e^{-C_1n\log^2(f(n))}.
$$
The last part of the proof coincides with that of Lemma \ref{LeG}.
\end{proof}

\begin{lemma}
    \label{Ldeg}
    Let $\{G_n\}$ satisfy the conditions of Theorem \ref{ProbMain}, let $k$ be a positive integer and let $p = p(n) = f(n)\frac{1}{\sqrt{d}}, f(n) \to \infty$. Then almost surely the following inequality holds for all sufficiently large $n$. 
    \begin{multline*}
    \forall V \subset V_{G_n}, |V| \geq \frac{n}{k} : \left|\sum_{v \in V}{\deg(v, p)} - p\sum_{v \in V}{\deg(v)}\right| < \frac{p\log(f(n))}{f(n)}\sum_{v \in V}{\deg(v)}  = \\ = \frac{p\log(f(n))}{f(n)}d|V|.
    \end{multline*}
\end{lemma}

\begin{proof}
Let us fix some $n$ and bound the probability of the following event
$$
\left|\sum_{v \in V}{\deg(v, p)} - p\sum_{v \in V}{\deg(v)}\right| \geq \frac{p\log(f(n))}{f(n)}\sum_{v \in V}{\deg(v)}
$$ 
for some fixed $V \subset V_{G_n}$. First let us rewrite $\sum_{v \in V}{\deg(v, p)}$ in form $\sum\limits_{e \in E_1}{X_e} + \sum\limits_{e \in E_2}{2X_e}$, where $E_1$ is the set of edges with exactly 1 endpoint in $V$, $E_2$ is the set of edges with both of their endpoints in $V$, and $X_e$ is an indicator of an edge $e$ from $G_n$ appearing in $G_n(p)$. Then, applying Lemma \ref{Hf} with $$m = |E_1| + |E_2| \geq \frac12d|V|, a_i = 0, b_i = 2, \varepsilon = \frac{p\log(f(n))}{f(n)},$$ we get
$$
\mathbf{P}\left(\left|\sum_{v \in V}{\deg(v, p)} - p\sum_{v \in V}{\deg(v)}\right| \geq \frac{p\log(f(n))}{f(n)}\sum_{v \in V}{\deg(v)}\right) < 2\exp\left\{-\frac12\varepsilon^2(|E_1| + |E_2|)\right\} \leq 
$$
$$
\leq 2\exp\left\{-\frac14\frac{p^2\log^2(f(n))}{f^2(n)}d|V|\right\}.
$$

Considering only $V : |V| \geq \frac{n}{k}$ we obtain the following bound for such $V$:
$$
\mathbf{P}\left(\left|\sum_{v \in V}{\deg(v, p)} - p\sum_{v \in V}{\deg(v)}\right| \geq \frac{p\log(f(n))}{f(n)}\sum_{v \in V}{\deg(v)}\right) < 2\exp\left\{-C(k)\log^2(f(n))n\right\}.
$$

The probability that the property does not hold for at least one such $V$ can then be easily bounded from above:
$$
2^{n + 1}\exp\left\{-Cn\log^2(f(n))\right\} \leq \exp\left\{(n + 1)\ln2 - Cn\log^2(f(n))\right\} < e^{-C_1n\log^2(f(n))}.
$$

The last part of the proof coincides with that of Lemma \ref{LeG}.
\end{proof}

\subsubsection{Proof of Theorem \ref{ProbMain}}

{\bf Theorem. } {\it Let $S$ be an infinite subset of $\mathbb{N}$. Let $\{G_n\}_{n \in S}$ be a sequence of $(n, d(n), \lambda(n))$-graphs with $\limsup_{n \to \infty} \frac{\lambda(n)}{d(n)} < 1$ and let $p = p(n) \gg \frac1{\sqrt{d}}$. Consider a sequence of random graphs $\{G_{n}(p)\}$. Then
$$
\mathbf{P}\left(\limsup_{n \to \infty} |q^*(G_{n}) - q^{*}(G_{n}(p))| = 0\right) = 1.
$$}

Let us prove that almost surely $\limsup_{n \to \infty} |q^*(G_{n}) - q^{*}(G_{n}(p))| < \varepsilon$ for any $\varepsilon > 0$, which is obviously equivalent to the statement of the theorem. Since the properties from Lemmas \ref{LeG} --- \ref{Ldeg} hold almost surely, we can assume that the bounds from these lemmas always hold. 

Let us first prove the upper bound, $q^{*}(G_n) \leq q^{*}(G_n(p)) + \varepsilon$.

Let us fix some $k$ (we will specify its value later). Consider an optimal partition $\mathcal{A}$ of graph $G_n$. Let us first assume that each of the sets in $\mathcal{A}$ has less than $\frac{(k - 2)n}{k}$ elements. Let us then apply Construction \ref{algorithm} to $\mathcal{A}$ to obtain $\mathcal{A}^{\prime}$ and $\mathcal{A}_{big}$. Note that while $\mathcal{A}^{\prime}$ is not a partition, there exists a set $U$ such that $|U| < \frac{n}{k}$ and $\mathcal{A}^{\prime} \cup \{U\}$ is a partition. Let us add the elements of $U$ to one of the sets (we will call this set $V$) from $\mathcal{A}^{\prime}$. So we obtain some partition $\mathcal{A}^{\prime\prime}$. Then we can get the following bound on $q(\mathcal{A}^{\prime\prime})$.
\begin{multline*}
q(\mathcal{A}^{\prime\prime}) = \sum\limits_{A \in \mathcal{A}^{\prime\prime}}{\frac{e(A)}{e(G_n)}} - \sum\limits_{A \in \mathcal{A}^{\prime\prime}}{\frac{d^2|A|^2}{4e^2(G_n)}} = \sum\limits_{A \in \mathcal{A}_{big}\setminus\{V\}}{\frac{e(A)}{e(G_n)}} + \sum\limits_{A \in \mathcal{A}^{\prime\prime}\setminus\mathcal{A}_{big}}{\frac{e(A)}{e(G_n)}} - \sum\limits_{A \in \mathcal{A}^{\prime\prime}}{\frac{|A|^2}{n^2}} \geq \\ \geq \sum\limits_{A \in \mathcal{A}}{\frac{e(A)}{e(G_n)}} - \sum\limits_{A \in \mathcal{A}^{\prime\prime}}{\frac{|A|^2}{n^2}} \geq \sum\limits_{A \in \mathcal{A}}{\frac{e(A)}{e(G_n)}} - \sum\limits_{A \in \mathcal{A}_{big}\setminus\{V\}}{\frac{|A|^2}{n^2}} - \sum\limits_{A \in \mathcal{A}^{\prime}\setminus \mathcal{A}_{big}}{\frac{|A|^2}{n^2}} - \\ - \frac{|V|^2}{n^2} + \frac{|V^2|}{n^2} - \frac{|U\cup V|^2}{n^2} \geq \sum\limits_{A \in \mathcal{A}}{\frac{e(A)}{e(G_n)}} - \sum\limits_{A \in \mathcal{A}}{\frac{|A|^2}{n^2}} + \frac{|V|^2}{n^2} - \frac{|U\cup V|^2}{n^2} -\\- \sum\limits_{A \in \mathcal{A}^{\prime} \setminus \mathcal{A}_{big}}{\frac{|A|^2}{n^2}} = q(\mathcal{A}) +\frac{|V|^2}{n^2} - \frac{(|U| + |V|)^2}{n^2} - \sum\limits_{A \in \mathcal{A}^{\prime} \cup \{U\} \setminus \mathcal{A}_{big}}{\frac{|A|^2}{n^2}} \geq \\ \geq q(\mathcal{A}) - \frac4k - \frac1{k^2} = q^{*}(G_n) - \frac4k - \frac1{k^2}.
\end{multline*}

Now consider $\mathcal{A}^{\prime\prime}$ as a partition of a random graph $G_n(p)$. Note that
$
q^{*}(G_n(p)) \geq q(\mathcal{A}^{\prime\prime}, p).
$
Let us rewrite $q(\mathcal{A}^{\prime\prime}, p)$ as follows using Lemmas \ref{LeG} --- \ref{Ldeg}.
\begin{equation*}
q(\mathcal{A}^{\prime\prime}, p) = \left(\sum\limits_{A \in \mathcal{A}^{\prime\prime}}{\frac{e(A)}{e(G_n)}} - \sum\limits_{A \in \mathcal{A}^{\prime\prime}}{\frac{\left(\sum_{v\in A}{\deg(v)}\right)^2}{4e^2(G_n)}}\right)(1 + o(1)) = q(\mathcal{A}^{\prime\prime})(1 + o(1)).
\end{equation*}
We thus obtain $q(\mathcal{A}^{\prime\prime}, p) \geq (1 - o(1))(q^{*}(G_n) - \frac4k - \frac1{k^2})$ and thus $q^{*}(G_n) \leq q^{*}(G_n(p)) + \frac4k + \frac1{k^2} + o(1)$.

Let us now assume that $\mathcal{A}$ has a set of size at least $\frac{(k - 2)n}{k}$. In this case the modularity of $G_n$ does not exceed $\frac{4k - 4}{k^2} \leq \frac4k$.
At the same time the modularity of $G_n(p)$ is at least 0, so in this case we have $q^{*}(G_n) \leq q^{*}(G_n(p)) + \frac4k$. 

It remains to notice that taking $k$ such that $ \frac4k + \frac1{k^2} < \varepsilon$ completes the proof of the first part.

Let us now prove the second part of the theorem, $q^{*}(G_n) \geq q^{*}(G_n(p)) - \varepsilon$.

Consider an optimal partition $\mathcal{A}$ of graph $G_n(p)$. Assume it contais a set $A_i$ such that $|A_i| > \frac{(k - 1)n}{k}$. The modularity of $G_n(p)$ is then bounded by $$1 - \frac{|A_i|^2}{n^2}(1 + o(1)) \leq \frac2{k} + o(1).$$ Now assume that there are no such $A_i$ in $\mathcal{A}$. Let us then apply Construction \ref{algorithm} to obtain $\mathcal{A}^{\prime}$ and $\mathcal{A}_{big}$ from $\mathcal{A}$. Just like in the previous part, let $U$ be the set of vertices that did not fall into any of the sets of $\mathcal{A}^{\prime}$. It is obvious that $q^{*}(G_n(p))$ does not exceed  
$$
    	1 
    	- \frac12\sum\limits_{A \in \mathcal{A}^{\prime}}{\frac{e(A, \overline{A}, p)}{e(G_n(p))}} - \sum\limits_{A \in \mathcal{A}_{big}}\frac{\left(\sum_{v \in A}{\deg(v, p)}\right)^2}{4e^2(G_n(p))}.
$$
Applying Lemmas \ref{LeG} -- \ref{Ldeg} we obtain that for all big enough $n$
\begin{multline*}
    	q^{*}(G_n(p)) \leq  \left(1 - \frac12\sum\limits_{A \in \mathcal{A}^{\prime}}{\frac{e(A, \overline{A})}{e(G_n)}} - \sum\limits_{A \in \mathcal{A}_{big}}\frac{|A|^2}{n^2}\right)(1 + o(1)) = \\ = \left(1 - \frac12\sum\limits_{A \in \mathcal{A}^{\prime}\cup \{U\}}{\frac{e(A, \overline{A})}{e(G_n)}} + \frac{e(U,\overline{U})}{dn} - \sum\limits_{A \in \mathcal{A}_{big}}\frac{|A|^2}{n^2} - \sum\limits_{A \in \mathcal{A}^{\prime} \cup \{U\} \setminus \mathcal{A}_{big}}\frac{|A|^2}{n^2} + \sum\limits_{A \in \mathcal{A}^{\prime} \cup \{U\} \setminus \mathcal{A}_{big}}\frac{|A|^2}{n^2}\right)\times \\ \times(1 + o(1)) \leq \left(q^{*}(G_n) + \frac{e(U, \overline{U})}{dn} + \sum\limits_{A \in \mathcal{A}^{\prime} \cup \{U\} \setminus \mathcal{A}_{big}}\frac{|A|^2}{n^2}\right)(1 + o(1)) \leq \\ \leq q^{*}(G_n) + \frac{d + \lambda}{dk^2} + \frac{2}{k} + o(1) \leq q^{*}(G_n) + \frac{2}{k^2} + \frac{2}{k} + o(1).
\end{multline*}

It remains to notice that taking any $k$ such $ \frac2k + \frac2{k^2} < \varepsilon$ completes the proof of the second part of the theorem.

\subsubsection{Application to Johnson graphs}

Applying Theorem \ref{Pr1} to Johnson graphs we obtain the following upper bounds for their modularity.

\begin{theorem}
\label{PC3}
Let $p = p(n) = \omega\left(\binom{n - r}{r}^{-1/2}\right)$.
Then almost surely 
$$
\limsup_{n \to \infty} q^{*}(G_p(n, r, 0)) \leq \limsup_{n \to \infty} \frac{r}{n - r}.
$$
\end{theorem}

\begin{theorem}
\label{PC5}
Let $r = r(n)$ and $s = s(n)$ satisfy the inequality $(r - s)(n - 1) \geq r(n - r)$ for all $n$. Let $p = p(n) = \omega\left(\frac{1}{\sqrt{\binom{n - r}{r - s}\binom{r}{s}}}\right)$. Then almost surely 
$$
\limsup_{n \to \infty} q^{*}(G_p(n, r, 0)) \leq \limsup_{n \to \infty} \frac{|sn - r^2|}{r(n - r)}.
$$
\end{theorem}

\begin{theorem}
\label{PC2}
Let $p = p(s) = \omega\left(\frac{\sqrt{s}}{2^{2s}}\right)$. Then almost surely
$$q^{*}(G_p(4s, 2s, s)) \to 0, n \to \infty.$$
\end{theorem}

\begin{theorem}
\label{PC1}
Let $r > s \geq 1$, $p = p(n) = \omega\left(n^{-\frac{r - s}{2}}\right)$. Then almost surely 
$$\limsup\limits_{n \to \infty}{q^*(G_p(n, r, s))} \leq \frac{s}{r}.$$
\end{theorem}

\begin{theorem}
\label{PC4}
Let $n \gg r(n) \gg s(n)$, $p = p(n) = \omega\left(\frac{1}{\sqrt{\binom{n - r}{r - s}\binom{r}{s}}}\right)$. Then almost surely 
$$
q^*(G_p(n, r, s)) \to 0, n \to \infty.
$$
\end{theorem}

Finally, the following lower bound on the modularity of $G(n, r, s)$ is an immediate corollary of Theorems \ref{ProbMain} and \ref{P4}.

\begin{theorem}
\label{PC6}
Let $r > s \geq 1$, $p = \omega\left(n^{-\frac{r-s}2}\right)$. Then almost surely 
$$\liminf_{n \to \infty}{q^*(G_p(n, r, s))} \geq \frac{s}{2r - s}\left(1 + \left(\frac{r - s}{r}\right)^{\frac{2r}{s}}\right).$$
\end{theorem}

\subsection{Giant component in Johnson graphs}

In \cite{GComp} (see Theorem 1) the following result was obtained.

\begin{theorem}
\label{GCGeneral}
Let $G_n$ be a sequence of $(n, d, \lambda)$-graphs with $\lambda = o(d)$, $n \to \infty$.
\begin{enumerate}
    \item For any $0 < \alpha < 1$ w.h.p. all components of $G_n(\frac{\alpha}{d})$ have size $O(\log(n))$.
    \item For any $\alpha > 1$ w.h.p. there exists a component of size $(1 + o(1))(1 - \frac{\overline{\alpha}}{\alpha})n$ in $G_n(\frac{\alpha}{d})$. Here, $\overline{\alpha}$ is a solution of $xe^{-x} = \alpha e^{-\alpha}$, different from $\alpha$. At the same time, all other components of $G_n(\frac{\alpha}{d})$ have size $O(\log(n))$.
\end{enumerate}
\end{theorem}

This means that our results on the spectrum of Johnson graphs also imply the $``$giant component$"$ theorems. For example, the following result, obtained in \cite{Yar1, Yar2, Yar3} trivially flows from Theorem \ref{Th2}.

\begin{theorem}
\label{GC421}
Let $G_s = G(4s, 2s, s), G_s(p_s) = G_{p_s}(4s, 2s, s)$, $d_s = d(G_s)$, $N_s = \binom{4s}{2s}$. Then
\begin{enumerate}
    \item For any $0 < \alpha < 1$ w.h.p. all components of $G_s(\frac{\alpha}{d_s})$ have size $O(\log(N_s))$.
    \item For any $\alpha > 1$ w.h.p. there exists a component of size $(1 + o(1))(1 - \frac{\overline{\alpha}}{\alpha})N_s$ in $G_s(\frac{\alpha}{d_s})$. Here, $\overline{\alpha}$ is a solution of $xe^{-x} = \alpha e^{-\alpha}$, different from $\alpha$. At the same time, all other components of $G_s(\frac{\alpha}{d_s})$ have size $O(\log(N_s))$.
\end{enumerate}
\end{theorem}

Similar theorems can be obtained for other parameters.

\begin{theorem}
\label{CorKG2}
Let $n \gg r(n)$. Denote $G_n = G(n, r, 0)$, $G_n(p_n) = G_p(n, r, 0)$, $d_n = d(G(n, r, 0))$, $N_n = \binom{n}{r}$. Then
\begin{enumerate}
    \item For any $0 < \alpha < 1$ w.h.p. all components of $G_n(\frac{\alpha}{d_n})$ have size $O(\log(N_n))$.
    \item For any $\alpha > 1$  w.h.p. there exists a component of size $(1 + o(1))(1 - \frac{\overline{\alpha}}{\alpha})N_n$ in $G_n(\frac{\alpha}{d_n})$. Here, $\overline{\alpha}$ is a solution of $xe^{-x} = \alpha e^{-\alpha}$, different from $\alpha$. At the same time, all other components of $G_n(\frac{\alpha}{d_n})$ have size $O(\log(N_n))$.
\end{enumerate}
\end{theorem}

\begin{theorem}
\label{Cor3}
Let $n \gg r(n) \gg s(n)$. Denote $G_n = G(n, r, s), G_n(p_n) = G_p(n, r, s)$, $d_n = d(G(n, r, s))$, $N_n = \binom{n}{r}$. Then
\begin{enumerate}
    \item For any $0 < \alpha < 1$ w.h.p. all components of $G_n(\frac{\alpha}{d_n})$ have size $O(\log(N_n))$.
    \item For any $\alpha > 1$ w.h.p. there exists a component of size $(1 + o(1))(1 - \frac{\overline{\alpha}}{\alpha})N_n$ in $G_n(\frac{\alpha}{d_n})$. Here, $\overline{\alpha}$ is a solution of $xe^{-x} = \alpha e^{-\alpha}$, different from $\alpha$. At the same time, all other components of $G_n(\frac{\alpha}{d_n})$ have size $O(\log(N_n))$.
\end{enumerate}
\end{theorem}

\subsection{Hamiltonicity}

In this subsection we obtain a new result about the existence of a cycle that passes through each vertex of $G(n, r, s)$ exactly once (also known as the hamiltonicity of $G(n, r, s)$ graphs). Before we start, let us introduce two parameters of a graph that help us in this proof.

\begin{definition}
Vertex connectivity of a graph $G$ is the minimum number of vertices we need to delete from $G$ to make the remaining graph disconnected (or $|V_G| - 1$ if we cannot obtain a disconnected graph by deleting some vertices of $G$). The notation is $\kappa(G)$.
\end{definition}

\begin{definition}
Independence number of a graph $G$ is the maximum number of vertices in $G$ such that there is no edge between these vertices. The notation is $\alpha(G)$.
\end{definition}

First of all, let us prove the following general bound on the vertex connectivity of pseudo-random graphs. 

\begin{theorem}\label{VerCon}
Let $G$ be an $(n, d, \lambda)$-graph with $\frac{(1 + d/n)^2}{2(1 - d/n)}\frac{\lambda^2}{d^2} + \frac{\lambda}{d} + \frac{d}{n} \leq 1$. Then $\kappa(G) \geq (1 - \frac{\lambda^2}{d^2\alpha^2})d$, where $\alpha$ is a (unique) positive solution of the equation $\frac{(1 + d/n)^2}{2(1 - d/n)}x^2 + x + \frac{d}{n} + \sqrt{x^2 - \frac{\lambda^2}{d^2}} = 1$.
\end{theorem}

Note, that this is a direct generalization of Theorem 4.1 from \cite{Kriv06}.

\begin{theorem}\label{VerConP}
Let $G$ be an $(n, d, \lambda)$-graph with $d \leq \frac{n}2, \lambda \leq \frac{d}{6}$. Then $\kappa(G) \geq (1 - \frac{36\lambda^2}{d^2})d$.
\end{theorem}

Indeed, for such parameters we have $\frac{(1 + d/n)^2}{2(1 - d/n)}\frac{\lambda^2}{d^2} + \frac{\lambda}{d} + \frac{d}{n} \leq \frac{1}{16} + \frac{1}{6} + \frac{1}{2} \leq 1$, so Theorem \ref{VerCon} is applicable. Moreover, since the function $f(x) = \frac{(1 + d/n)^2}{2(1 - d/n)}x^2 + x + \frac{d}{n} + \sqrt{x^2 - \frac{\lambda^2}{d^2}}$ is monotone and $f(\frac{1}{6}) \leq \frac{1}{16} + \frac{1}{6} + \frac{1}{2} + \frac{1}{6} \leq 1$, we obtain that the bound from Theorem \ref{VerCon} is stronger than that of Theorem \ref{VerConP} in this case.

In the proof we will use the following theorem that helps us in estimating the number of edges between two sets of vertices (see \cite{Kriv06}, Theorem 2.11).

\begin{theorem}
\label{MainL}
Let $A$ and $B$ be two subsets of vertices in an $(n, d, \lambda)$-graph. Then 
$$
\left|e(A, B) - \frac{d}{n}|A||B|\right| \leq \lambda\sqrt{|A||B|\left(1-\frac{|A|}{n}\right)\left(1-\frac{|B|}{n}\right)},
$$
where $e(A, B)$ is the number of edges connecting a vertex from $A$ with a vertex from $B$ (edges connecting vertices in $A \cap B$ are counted twice).
\end{theorem}

\subsubsection{Proof of Theorem \ref{VerCon}}

{\bf Theorem. }{\it Let $G$ be an $(n, d, \lambda)$-graph with $\frac{(1 + d/n)^2}{2(1 - d/n)}\frac{\lambda^2}{d^2} + \frac{\lambda}{d} + \frac{d}{n} \leq 1$. Then $\kappa(G) \geq (1 - \frac{\lambda^2}{d^2\alpha^2})d$, where $\alpha$ is a (unique) positive solution of the equation $\frac{(1 + d/n)^2}{2(1 - d/n)}x^2 + x + \frac{d}{n} + \sqrt{x^2 - \frac{\lambda^2}{d^2}} = 1$.}

First of all, we may assume that $\alpha d > \lambda$, as in the other case there is nothing to prove.

Assume the contrary, then there exists a set of vertices $S, |S| <  d - \frac{\lambda^2}{d\alpha^2}$, such that if we remove this set, the graph becomes disconnected. Let $U$ be the smallest component in the remaining graph, $W = V_G \setminus (U \cup S)$. Then $|W| \geq \frac{n - d}{2}$. Moreover, $|U| + |S| > d$, so $|U| > \frac{\lambda^2}{d\alpha^2}$.

Since $G[V_G \setminus S]$ is disconnected, $$d|U||W| \leq \lambda\sqrt{|U||W|(n - |U|)(n - |W|)}.$$ From AM-GM we have $(n - |U|)(n - |W|) \leq \frac{(n - |U| + n - |W|)^2}{4} = \frac{(n + |S|)^2}{4} \leq \frac{(n + d)^2}{4}$, so we can further rewrite this inequality as
$$
d|U||W| \leq \lambda\frac{n + d}{2}\sqrt{|U||W|}.
$$
From this we have $|U| \leq \frac{\lambda n}{d} \frac{\lambda}{d} \frac{(n + d)^2}{4n|W|} < \frac{\alpha\lambda n}{d} \frac{(n + d)^2}{2n(n - d)}$.

Let us now bound $e(U)$ with the help of Theorem \ref{MainL}.

\begin{multline*}
e(U) = \frac{1}{2}e(U, U) \leq \frac{d|U|^2}{2n} + \frac{\lambda}{2}|U|\left(1 - \frac{|U|}{n}\right) \leq \\ \leq \frac{d|U|^2}{2n} + \frac{\lambda|U|}{2} < \frac{\alpha\lambda n}{d}\frac{(n + d)^2}{2n(n - d)}\frac{d|U|}{2n} + \frac{\lambda|U|}{2} = \\ = \frac{ \alpha(n + d)^2}{2n(n - d)}\frac{\lambda |U|}{2} + \frac{\lambda|U|}{2} = \frac{\alpha(n + d)^2 + 2n(n - d)}{4n(n - d)}\lambda|U| \leq \\ \leq \frac{\alpha^2(n + d)^2 + 2\alpha n(n - d)}{4n(n - d)}d|U|.
\end{multline*}

Finally, we apply Theorem \ref{MainL} to $e(U, \overline{U}) = e(U, S)$ to get the following inequality.
\begin{multline*}
    e(U, \overline{U}) = e(U, S) \leq \frac{d|U||S|}{n} + \lambda\sqrt{|U||S|} = \left(\frac{|S|}{n} + \frac{\lambda}{d}\sqrt{\frac{|S|}{|U|}}\right)d|U| < \\ 
    < \left(\frac{|S|}{n} + \sqrt{\frac{\alpha^2|S|}{d}}\right)d|U| \leq \left(\frac{d}{n} + \sqrt{\frac{\alpha^2d - \lambda^2/d}{d}}\right)d|U| =
    \\ = \left(\frac{d}{n} + \sqrt{\alpha^2 - \frac{\lambda^2}{d^2}}\right)d|U|.
\end{multline*}

Let us now combine these bounds to obtain
\begin{multline*}
d|U| = 2e(U) + e(U, \overline{U}) < \frac{\alpha^2(n + d)^2 + 2\alpha n(n - d)}{2n(n - d)}d|U| + \left(\frac{d}{n} + \sqrt{\alpha^2 - \frac{\lambda^2}{d^2}}\right)d|U| = \\ = \left(\frac{\alpha^2(n + d)^2}{2n(n - d)} + \alpha + \frac{d}{n} + \sqrt{\alpha^2 - \frac{\lambda^2}{d^2}}\right)d|U| = d|U|,
\end{multline*}

which is obviously a contradiction.

\subsubsection{Application to Johnson graphs}

Let us prove the hamiltonicity of $G(n, r, s)$ graphs for all fixed $r \geq 2s + 1$ and all sufficiently large $n$.

\begin{theorem}
\label{T320}
Let $r \geq 2s + 1$. There exists a constant $N$ such that for all $n > N$ we have $G(n, r, s)$ are hamiltonian.
\end{theorem}

In the proof of Theorem \ref{T320} we use the following bound on the independence number of $G(n, r, s), r \geq 2s + 1$, which was obtained in \cite{Fur}.

\begin{theorem}
\label{FF}
For $r \geq 2s + 1$ 
$$
\alpha(G(n, r, s)) \leq \binom{n - s - 1}{r - s - 1}.
$$
\end{theorem}

The other theorem we use is the sufficient condition of hamiltonicity proved by Chv\'atal and Erd\H os (see \cite{ErdChv}).

\begin{theorem}
\label{CE}
Let $G$ be a graph on $n \geq 3$ vertices such that $\kappa(G) \geq \alpha(G)$. Then $G$ contains a Hamilton cycle.
\end{theorem}

\begin{proof}[Proof of Theorem \ref{T320}]
In the proof we will use the notation $N_V \coloneqq \binom{n}{r}, d \coloneqq \binom{n - r}{r - s}\binom{r}{s}$.

Let us take such $N$ that for all $n > N$
\begin{enumerate}
    \item $\binom{n - r}{r - s}\binom{r}{s} \leq \frac{1}{20}\binom{n}{r} \Leftrightarrow d \leq \frac{N_V}{20}$;
    \item $\lambda(G(n, r, s)) = |E_{r - s}(1)|$;
    \item $\left(1 - \frac{1}{1.01^2}\right)\binom{n - r}{r - s}\binom{r}{s} \geq \binom{n - s - 1}{r - s - 1} \Rightarrow \left(1 - \frac{1}{1.01^2}\right)d \geq \alpha(G(n, r, s))$.
    \item $r^2 < (r - s)n \Leftrightarrow \frac{s}{n - r} < \frac{r}{n}$.
    \item $sn > r^2 \Leftrightarrow \frac{s}{r} > \frac{r}{n}$.
\end{enumerate}
Then for any $n > N$ we know that $\lambda(G(n, r, s)) = |E_{r - s}(1)| = \frac{|sn -  r^2|}{r(n - r)}d = \left|\frac{s}{r} + \frac{s}{n - r} - \frac{r}{n}\right|d \leq \frac{d}{2}$ (since $\left|\frac{s}{r} + \frac{s}{n - r} - \frac{r}{n}\right| \in \left[0, \frac{s}{r}\right]$). We thus have $$\frac{(1 + d/N_V)^2}{2(1 - d/N_V)}\frac{\lambda^2}{d^2} + \frac{\lambda}{d} + \frac{d}{N_V} \leq \frac{(1 + 1/20)^2}{2(1 - 1/20)}\frac{1}{4} + \frac{1}{2} + \frac{1}{20} \leq 1.$$

We now want to obtain a bound on the positive solution of the equation $f(x) = 1$, where $$
f(x) = \frac{(1 + d/N_V)^2}{2(1 - d/N_V)}x^2 + x + \frac{d}{N_V} + \sqrt{x^2 - \frac{\lambda^2}{d^2}}.
$$

Let us calculate the value $f\left(\frac{1.01\lambda}{d}\right)$:

$$
f\left(\frac{1.01\lambda}{d}\right) = \frac{(1 + d/N_V)^2}{2(1 - d/N_V)}\frac{1.01^2\lambda^2}{d^2} + \frac{1.01\lambda}{d} + \frac{d}{N_V} + \frac{\sqrt{1.01^2 - 1}\lambda}{d} \leq
$$
$$
\leq \frac{(1 + 1/20)^2}{2(1 - 1/20)}\frac{1.01^2}{4} + \frac{1.01}{2} + \frac{1}{20} + \frac{\sqrt{1.01^2 - 1}}{4} < 1.
$$

Since $f$ is monotone, we can get that 
$$
\kappa(G(n, r, s)) \geq \left(1 - \frac{\lambda^2}{d^2\alpha^2}\right)d,
$$
where $\alpha$ is at least $\frac{1.01\lambda}{d}$. This means that 
$$
\kappa(G(n, r, s)) \geq \left(1 - \frac{1}{1.01^2}\right)d,
$$

It remains to use Theorem \ref{FF} to obtain $\kappa(G(n, r, s)) \geq \left(1 - \frac{1}{1.01^2}\right)d \geq \binom{n - s - 1}{r - s - 1} \geq \alpha(G(n, r, s))$ and then apply the Chv\'atal-Erd\"os condition of hamiltonicity.

\end{proof}

\subsection{The threshold for the appearance of a Hamilton cycle in Johnson graphs}

Theorem \ref{Th2} also allows us to prove the following result on the threshold for the appearance of a Hamilton cycle in $G(4s, 2s, s)$.

\begin{theorem}
\label{HamG} \text{}
    Denote $N := V_{G(4s,2s,s)} = \binom{4s}{2s}, d := d(G(n, r, s)) = \binom{2s}{s}^2$.
\begin{enumerate}
    \item Let $p = \frac{\ln N + \ln\ln N + w(N)}{d}, w(N) \to \infty$. Then w.h.p. $G_p(4s, 2s, s)$ is hamiltonian. 
    \item Let $p = \frac{\ln N + \ln\ln N - w(N)}{d}, w(N) \to \infty$. Then w.h.p. $G_p(4s, 2s, s)$ is not hamiltonian.
\end{enumerate}
\end{theorem}

In the proof we will use a theorem from \cite{Kr}. But first let us introduce some definitions.

\begin{definition}
$\{G_t\}$ is called a random graph process with base graph $G$ if $G_0$ is a graph with $V_{G_0} = V_G, E_{G_0} = \varnothing$ and $G_t$ is obtained from $G_{t - 1}$ by adding an edge from $E_{G} \setminus E_{G_{t - 1}}$ to $G_{t}$. The edge is chosen equiprobably.
\end{definition}

\begin{definition}
Let $\{G_t\}$ be a random graph process with base graph $G$. Then $\tau_{A_{2k}}$ is equal to minimum $t$ such that $G_t$ contains $k$ edge-disjoint Hamilton cycles.
\end{definition}

\begin{definition}
Let $\{G_t\}$ be a random graph process with base graph $G$. Then $\tau_{2k}$ is equal to minimum $t$ such that the degrees of all vertices in $G_t$ are at least $2k$.
\end{definition}

\begin{theorem}[\cite{Kr}]
\label{HamDense}
Let $C = 10^8,c = \frac{1}{400}, d = d(n) \geq \frac{Cn\log\log n}{\log n}, \lambda = \lambda(n) \leq \frac{cd^2}{n}, k \in \mathbb{N}$, let also $G$ be an $(n, d, \lambda)$-graph, and $\{G_t\}$ be a random graph process with base graph $G$. Then w.h.p. $\tau_{2k}(\{G_t\}) = \tau_{A_{2k}}(\{G_t\})$.
\end{theorem}

\subsubsection{Proof of Theorem \ref{HamG}}

{\bf Theorem. } {\it
Denote $N := V_{G(4s,2s,s)} = \binom{4s}{2s}, d := d(G(n, r, s)) = \binom{2s}{s}^2$.
\begin{enumerate}
    \item Let $p = \frac{\ln N + \ln\ln N + w(N)}{d}, w(N) \to \infty$. Then w.h.p. $G_p(4s, 2s, s)$ is hamiltonian. 
    \item Let $p = \frac{\ln N + \ln\ln N - w(N)}{d}, w(N) \to \infty$. Then w.h.p. $G_p(4s, 2s, s)$ is not hamiltonian.
\end{enumerate}}

First of all notice that $$d = d(G(4s, 2s, s)) = \binom{2s}{s}^2 \sim \frac{2^{4s}}{\pi s},$$ $$N = V(G(4s, 2s, s)) = \binom{4s}{2s} \sim \frac{2^{4s}}{\sqrt{2\pi s}}.$$ We thus get that $\frac{CN\log\log N}{\log N} \sim \frac{C2^{4s}\log s}{\sqrt{2\pi s}4s} \ll \frac{2^{4s}}{\pi s} \sim d$. The inequality $d \geq \frac{CN\log\log N}{\log N}$ is then satisfied for all big enough $s$. Let us now use theorem \ref{Th2} to show that $\lambda(G(4s, 2s, s)) = o(\frac{d^2}{N})$. Indeed, $\frac{d^2}{N} \sim \frac{C_02^{4s}}{s^{3/2}} = \frac{C_1d}{\sqrt{s}} = \omega(\frac{d}{4s}) = \omega(\lambda(G(4s, 2s, s)))$. We have thus shown that all conditions of Theorem \ref{HamDense} are satisfied. That means that for a random graph process with base graph $G(4s, 2s, s)$ the time of the appearance of a Hamilton cycle w.h.p. coincides with the time when the minimum degree of $G_t$ becomes equal to 2.

Let us now use the following well-known result. For all monotone properties $\mathcal{A}$ the following two statements, $P(G_t \in \mathcal{A}) \to 1$ and $P(G(t / N) \in \mathcal{A}) \to 1$, are equivalent. It thus remains to prove that for $$p_0 = \frac{\log N + \log\log N - w(N)}{d}, w(N) \to \infty$$ w.h.p. there exists a vertex in $G_{p_0}(4s, 2s, s)$ of degree 0 or 1, while for $$p_1 = \frac{\log N + \log\log N + w(N)}{d}, w(N) \to \infty$$ w.h.p. there is no such vertex in $G_{p_1}(4s, 2s, s)$.

Let us move to the proof of these statements. Let us first count $$e_p = \mathbf{E}(\{\text{number of vertices of degree} \leq 1 \text{ in } G_p(4s, 2s, s)\}).$$ 
$$
e_p = N((1-p)^d + d(1 - p)^{d-1}p) = N(1-p)^{d-1}(1 + (d - 1)p).
$$
Let us now substitute $p$ with $p_0$:
\begin{multline*}
e_{p_0} = N\left(1 - \frac{\log N + \log\log N - w(N)}{d}\right)^{d-1}\left(1 + (d-1)\frac{\log N + \log\log N - w(N)}{d}\right) \geq \\
\geq N\log N\left(1 - \frac{\log N + \log\log N - w(N)}{d}\right)^{d-1} =  \\
= \exp\left\{\log{N} + \log\log N + (d-1)\log\left(1 - \frac{\log N + \log\log N - w(N)}{d}\right)\right\} = \\
= \exp\left\{\log{N} + \log\log N - d\frac{\log N + \log\log N - w(N)}{d} - O\left(\frac{(\log N + \log\log N - w(N))^2}{d}\right)\right\} = \\
= \exp\{w(N) - o(1)\} \to \infty.
\end{multline*}
Similarly, for $p_1$ we have
\begin{multline*}
e_{p_1} = N\left(1 - \frac{\log N + \log\log N + w(N)}{d}\right)^{d-1}\left(1 + (d-1)\frac{\log N + \log\log N + w(N)}{d}\right) \geq \\
\geq N\log N\left(1 - \frac{\log N + \log\log N + w(N)}{d}\right)^{d-1} =  \\
= \exp\left\{\log{N} + \log\log N + (d-1)\log\left(1 - \frac{\log N + \log\log N + w(N)}{d}\right)\right\} = \\
= \exp\left\{\log{N} + \log\log N - d\frac{\log N + \log\log N + w(N)}{d} + O\left(\frac{(\log N + \log\log N + w(N))^2}{d}\right)\right\} = \\
= \exp\{-w(N) + o(1)\} \to 0.
\end{multline*}

So for $G_{p_1}(4s, 2s, s)$ the expected number of vertices with degree 1 or 0 tends to 0. This means that $\mathbf{P}(\exists v \in G_{p_1}(4s, 2s, s) : \deg(v) \geq 2) \leq e_{p_1} \to 0$. For the second part of the proof let us calculate the second factorial moment $e_2$ of the number of vertices with degree 1 or 0 in $G_{p_0}(4s, 2s, s)$. Trivially, 
$$
e_2 = \sum\limits_{u, v \in G_{p_0}(4s, 2s, s), u \neq v}{P(u \text{ and } v \text{ both have degree } \leq 1)}.
$$

The value of each term of the sum depends on the existence of an edge between the respective vertices in $G(4s, 2s, s)$. If $u$ and $v$ are not connected in $G(4s, 2s, s)$, the probability is equal to the square of the probability $P(v \text{ has degree } \leq 1)$, so the corresponding term is equal to $(1 - p_0)^{2(d-1)}(1 + (d - 1)p_0)^2$. It remains to bound the probability in case there is an edge between $u$ and $v$ in $G(4s, 2s, s)$. In this case the probability is equal to
\begin{multline*}
(1 - p_0)((1 - p_0)^{d-1} + (d-1)p_0(1 - p_0)^{d-2})^{2} + p_0(1-p_0)^{2(d-1)} = \\ =
(1-p_0)^{2(d-1) - 1}(1 + (d - 2)p_0)^2 + p_0(1 - p_0)^{2(d-1)} = \\ =
(1 - p_0)^{2(d-1)}\left(p_0 + \frac{(1 + (d-2)p_0)^2}{1 - p_0}\right).
\end{multline*}
Let us prove that this probability is $(1 - p_0)^{2(d-1)}(1 + (d - 1)p_0)^2(1 + o(1))$:
\begin{equation*}
    \frac{(1 - p_0)^{2(d-1)}\left(p_0 + \frac{(1 + (d-2)p_0)^2}{1 - p_0}\right)}{(1 - p_0)^{2(d-1)}(1 + (d - 1)p_0)^2} = \frac{\left(p_0 + \frac{(1 + (d-2)p_0)^2}{1 - p_0}\right)}{(1 + (d - 1)p_0)^2} = \frac{\left(o(1) + \frac{d^2p_0^2(1 + o(1))}{1 - o(1)}\right)}{d^2p_0^2(1 + o(1))} = 1 + o(1).
\end{equation*}
We have thus proved that every term in $e_2$ is $(1 - p_0)^{2(d-1)}(1 + (d - 1)p_0)^2(1 + o(1))$, and so $e_2 \leq e_{p_0}^2(1 + o(1))$. It remains to apply Chebyshev inequality:
$$
\mathbf{P}(\forall v : \deg v > 1) \leq \frac{e_2 + e_{p_0} - e_{p_0}^2}{e_{p_0}^2} = \frac{o(e_{p_0}^2)}{e_{p_0}^2} = o(1).
$$

\section*{Appendix 1}

\begin{proof}[Proof of Lemma \ref{exp2}]
Consider the following 3 cases:
\begin{enumerate}
    \item $j = 0$ (this means that $i \leq s$). We want to prove that $C_{n, r, s}(i, 0) = \frac{(s - i + 1)(n - 2r + s - i + 1)}{(r - i + 1)(n - r - i + 1)}C_{n, r, s}(i - 1, 0)$. Using $C_{n, r, s}(i, 0) = \binom{r - i}{r - s}\binom{n - r - i}{r - s}$ we obtain 
    \begin{equation*}
        \binom{r - i}{r - s}\binom{n - r - i}{r - s} = \frac{(s - i + 1)(n - 2r + s - i + 1)}{(r - i + 1)(n - r - i + 1)}\binom{r - i + 1}{r - s}\binom{n - r - i + 1}{r - s},
    \end{equation*}
    which is obviously true.
    
    \item $j = i$ (this means that $i \leq r - s$).  We want to prove that $C_{n, r, s}(i, i) = \frac{(r - s - i + 1)^2}{(r - i + 1)(n - r - i + 1)}C_{n, r, s}(i - 1, i - 1)$. Using $C_{n, r, s}(i, i) = \binom{r - i}{r - s - i}\binom{n - r - i}{r - s - i}$ we obtain
    $$
        \binom{r - i}{r - s - i}\binom{n - r - i}{r - s - i} = \frac{(r - s - i + 1)^2}{(r - i + 1)(n - r - i + 1)}\binom{r - i + 1}{r - s - i + 1}\binom{n - r - i + 1}{r - s - i + 1},
    $$ which is obviously true.
    
    \item $0 < j < i$, $j \geq i - s$. 
    \begin{multline*}
        \frac{(r - s - j + 1)^2}{(r - i + 1)(n - r - i + 1)}C_{n, r, s}(i - 1, j - 1) + \\ + \frac{(s - i + j + 1)(n - 2r + s + j - i + 1)}{(r - i + 1)(n - r - i + 1)}C_{n, r, s}(i - 1, j) =
        \\
        = \frac{(r - s - j + 1)^2}{(r - i + 1)(n - r - i + 1)}\binom{i - 1}{j - 1}\binom{r - i + 1}{r - s - j + 1}\binom{n - r - i + 1}{r - s - j + 1} + \\ + \frac{(s - i + j + 1)(n - 2r + s + j - i + 1)}{(r - i + 1)(n - r - i + 1)}\binom{i - 1}{j}\binom{r - i + 1}{r - s - j}\binom{n - r - i + 1}{r - s - j} =
        \\ 
        = \binom{i - 1}{j - 1}\binom{r - i}{r - s - j}\binom{n - r - i}{r - s - j} + \binom{i - 1}{j}\binom{r - i}{r - s - j}\binom{n - r - i}{r - s - j} =\\= \binom{i}{j}\binom{r - i}{r - s - j}\binom{n - r - i}{r - s - j} = C_{n, r, s}(i, j).
    \end{multline*}
\end{enumerate}
\end{proof}

\section*{Appendix 2}

Here we give the exact values of $E_{s}(4), E_{s}(6), E_{s}(8)$ and $E_{s}(10)$ for the symmetric case ($n = 2r = 4s$).

\begin{enumerate}
    \item 
    \begin{multline*}
        E_{s}(4) = \binom{2s - 4}{s}^2 - 4\binom{2s - 4}{s - 1}^2 + 6\binom{2s - 4}{s - 2}^2 - 4\binom{2s - 4}{s - 3}^2 + \binom{2s - 4}{s - 4}^2 = \\ = \left(6 - 8\frac{(s - 2)^2}{(s - 1)^2} + 2\frac{(s - 3)^2(s - 2)^2}{s^2(s - 1)^2} \right)\binom{2s - 4}{s - 2}^2 = \\ = \left(\frac{6s^2(s - 1)^2 - 8(s-2)^2s^2 + 2(s - 3)^2(s - 2)^2}{s^2(s - 1)^2} \right)\binom{2s - 4}{s - 2}^2 = \\ = \left(\frac{48s^2 - 120 s + 72}{s^2(s - 1)^2} \right)\binom{2s - 4}{s - 2}^2 \sim \frac{3}{16s^2}\frac{2^{4s}}{\pi s} = \frac{3}{16s^2}E_{s}(0).
    \end{multline*}
    
    \item
    \begin{multline*}
    E_{s}(6) = \binom{2s - 6}{s}^2 - 6\binom{2s - 6}{s - 1}^2 + 15\binom{2s - 6}{s - 2}^2 - 20\binom{2s - 6}{s - 3}^2 + \\ + 15\binom{2s - 6}{s - 4}^2 - 6\binom{2s - 6}{s - 5}^2 + \binom{2s - 6}{s - 6}^2 = \\ = \left(2\frac{(s - 3)^2(s - 4)^2(s - 5)^2}{s^2(s - 1)^2(s - 2)^2} - 12\frac{(s - 3)^2(s - 4)^2}{(s - 1)^2(s - 2)^2} + 30\frac{(s - 3)^2}{(s - 2)^2} - 20\right)\binom{2s - 6}{s - 3}^2 = \\ = 
    \left(\frac{-960s^3 + 5760s^2 - 11280s + 7200}{s^2(s - 1)^2(s - 2)^2}\right)\binom{2s - 6}{s - 3}^2 \sim \\ \sim -\frac{15}{64s^3}\frac{2^{4s}}{\pi s} = -\frac{15}{64s^3}E_{s}(0).
    \end{multline*}
    
    \item
    \begin{multline*}
    E_{s}(8) = \binom{2s - 8}{s}^2 - 8\binom{2s - 8}{s - 1}^2 + 28\binom{2s - 8}{s - 2}^2 - 56\binom{2s - 8}{s - 3}^2 + 70\binom{2s - 8}{s - 4}^2 - \\ - 56\binom{2s - 8}{s - 5}^2 + 28\binom{2s - 8}{s - 6}^2 - 8\binom{2s - 8}{s - 7}^2 + \binom{2s - 8}{s - 8}^2 = \\ = \left(2\frac{(s - 4)^2(s - 5)^2(s - 6)^2(s - 7)^2}{s^2(s - 1)^2(s - 2)^2(s - 3)^2} - 16\frac{(s - 4)^2(s - 5)^2(s - 6)^2}{(s - 1)^2(s - 2)^2(s - 3)^2} + \right. \\ \left. + 56\frac{(s - 4)^2(s - 5)^2}{(s - 2)^2(s - 3)^2} - 112\frac{(s - 4)^2}{(s - 3)^2} + 70\right)\binom{2s - 8}{s - 4}^2 = \\ = 
    \left(\frac{26880s^4 - 295680s^3 + 1202880s^2 - 2143680s + 1411200}{s^2(s - 1)^2(s - 2)^2(s - 3)^2}\right)\binom{2s - 8}{s - 4}^2 \sim \\ \sim \frac{105}{2^{8}s^4}\frac{2^{4s}}{\pi s} = \frac{105}{256s^4}E_{s}(0).
    \end{multline*}
\end{enumerate}

\section*{Appendix 3}

Here we prove that for any $i, j$ such that $0 \leq i \leq \frac{r}{10}, 0 \leq j \leq \min\{i, r - s\}$ and $n > 100r > 10000s$ the following inequality holds:
$$
C_{n, r, s}(i, j) \leq \left(\frac12\right)^iC_{n, r, s}(0, 0).
$$

\begin{proof}
Let us prove this inequality by induction. The case of $i = 0$ is obvious, so we only need to prove the induction step. Let us make use of the formula from Lemma \ref{exp2}. 

\begin{multline*}
    C_{n, r, s}(i, j) = \frac{(r-s-j+1)^2}{(r - i + 1)(n - r - i + 1)}C_{n, r, s}(i - 1, j - 1) +\\+ \frac{(s - i + j + 1)(n - 2r + s + j - i + 1)}{(r - i + 1)(n - r - i + 1)}C_{n, r, s}(i - 1, j) \leq \\ \leq \left(\left|\frac{(r-s-j+1)^2}{(r - i + 1)(n - r - i + 1)}\right| + \left|\frac{(s - i + j + 1)(n - 2r + s + j - i + 1)}{(r - i + 1)(n - r - i + 1)}\right|\right)\left(\frac12\right)^{i - 1}C_{n, r, s}(0, 0). 
\end{multline*}

All that remains is to prove that 
$$
\left|\frac{(r-s-j+1)^2}{(r - i + 1)(n - r - i + 1)}\right| + \left|\frac{(s - i + j + 1)(n - 2r + s + j - i + 1)}{(r - i + 1)(n - r - i + 1)}\right| \leq \frac12,
$$
which can be checked by direct calculations.

\end{proof}

\section*{Appendix 4}

Let us use Lemma \ref{exp2} to obtain a new expression for $E_{r - s}(i)$.
\begin{multline*}
    E_{r - s}(i) = \sum\limits_{j = \max\{0, i - s\}}^{i}{(-1)^{j}C_{n, r, s}(i, j)} = 
    \\ 
    = \sum\limits_{j = \max\{0, i - s\}}^{i} (-1)^{j}\left(\frac{(r - s - j + 1)^2}{(r - i + 1)(n - r - i + 1)}C_{n, r, s}(i - 1, j - 1) + \right. \\ \left. + \frac{(s - i + j + 1)(n - 2r + s + j - i + 1)}{(r - i + 1)(n - r - i + 1)}C_{n, r, s}(i - 1, j)\right) =
    \\
    = \sum\limits_{j = \max\{0, i - s\} + 1}^{i} (-1)^{j} \left(\frac{(r - s - j + 1)^2}{(r - i + 1)(n - r - i + 1)} - \right. \\ \left. - \frac{(s - i + j)(n - 2r + s + j - i)}{(r - i + 1)(n - r - i + 1)}\right)C_{n, r, s}(i - 1, j - 1) + \\ +  (-1)^{\max\{0, i - s\}}\frac{(r - s - \max\{0, i - s\} + 1)^2}{(r - i + 1)(n - r - i + 1)}C_{n, r, s}(i - 1, \max\{0, i - s\} - 1).
\end{multline*}

\section*{Appendix 5}

Here we prove that in the case of $n \gg r \gg s, r = O(\sqrt{n}), i \geq 2s$ the following bound holds:
$$
\frac{\binom{n}{r}}{\left(\binom{n}{i} - \binom{n}{i - 1}\right)\binom{r}{s}\binom{n - r}{r - s}} = O\left(\frac{1}{n}\right).
$$

\begin{proof}
Let us rewrite the left part.
\begin{multline*}
	\frac{\binom{n}{r}}{\left(\binom{n}{i} - \binom{n}{i - 1}\right)\binom{r}{s}\binom{n - r}{r - s}} = \frac{i!(n - i + 1)!s!(r - s)!^2(n - 2r + s)!}{r!^2(n - r)!^2(n - 2i + 1)} \leq \\ 
	\leq \frac{(2s)!(n - 2s + 1)!s!(r - s)!^2(n - 2r + s)!}{r!^2(n - r)!^2(n - 2r + 1)} \leq \\ 
	\leq \frac{C}{n - 2r + 1}\exp\{(2s + 1/2)\ln(2s) + (n - 2s + 3/2)\ln(n - 2s + 1) + (s + 1/2)\ln s +\\+ (2r - 2s + 1)\ln(r - s) + (n - 2r + s + 1/2)\ln(n - 2r + s) - (2r + 1)\ln r - (2n - 2r + 1)\ln(n - r)\} = \\ 
    = \frac{C}{n - 2r + 1}\exp\{3s\ln s + (n - 2s + 3/2)\ln n + (2r - 2s + 1)\ln r +\\+ (n - 2r + s + 1/2)\ln(n - 2r + s) - (2r + 1)\ln r - (2n - 2r + 1)\ln(n - r) + O(s)\} = \\
	= \frac{C}{n - 2r + 1}\exp\{3s\ln s - (s - 1)\ln n - 2s\ln r +\\+ (n - 2r + s + 1/2)\ln(1 - (2r - s)/n) - (2n - 2r + 1)\ln(1 - r/n) + O(s)\} = \\
	= \frac{C}{n - 2r + 1}\exp\{3s\ln s - (s - 1)\ln n - 2s\ln r +\\+ (n - 2r + s)\ln(1 - (2r - s)/n) - (2n - 2r)\ln(1 - r/n) + O(s)\}.
\end{multline*}

Let us now use the equality $\ln(1 - x) = -x - x^2/2 + O(x^3)$:
\begin{multline*}
	\frac{\binom{n}{r}}{\left(\binom{n}{i} - \binom{n}{i - 1}\right)\binom{r}{s}\binom{n - r}{r - s}} \leq \\
	\leq \frac{C}{n - 2r + 1}\exp\{3s\ln s - (s - 1)\ln n - 2s\ln r +\\+ (n - 2r + s)\ln(1 - (2r - s)/n) - (2n - 2r)\ln(1 - r/n) + O(s)\} = \\
	= \frac{C}{n - 2r + 1}\exp\{3s\ln s - (s - 1)\ln n - 2s\ln r -\\- (n - (2r - s))((2r - s)/n + (2r - s)^2/2n^2 + O(r^3 / n^3)) + (2n - 2r)(r/n + r^2 / 2n^2 + O(r^3 / n^3)) + O(s)\} = \\
	= \frac{C}{n - 2r + 1}\exp\{3s\ln s - (s - 1)\ln n - 2s\ln r -\\- (2r - s) + (2r - s)^2 / n - (2r - s)^2 / 2n + 2r - 2r^2 / n + r^2 / n + O(s)\} = \\
	= \frac{C}{n - 2r + 1}\exp\{3s\ln s - (s - 1)\ln n - 2s\ln r + r^2 / n + O(s)\} = \\ = \frac{C}{n - 2r + 1}\exp\{3s\ln s - (s - 1)\ln n - 2s\ln r + O(s)\}.
\end{multline*}

Now consider the following two cases.
\begin{enumerate}

\item $s \leq \ln\ln r$. Then we can obtain the bound $\frac{C}{n - 2r + 1}\exp\{3\ln\ln r\ln\ln\ln r - 2\ln r + O(\ln\ln r)\} = O(1/n)$.

\item $s > \ln\ln r$. In this case we can use the fact that $\ln n > 2\ln r - C_1 \geq 2\ln s - C_1$ and thus $\frac{C}{n - 2r + 1}\exp\{3s\ln s - 2(s - 1)\ln s - 2s\ln s + O(s)\} = \frac{C}{n - 2r + 1}\exp\{-(s - 1)\ln s + O(s)\} = O(1/n)$.
\end{enumerate}	
\end{proof}

\section*{Acknowledgements}

I would like to thank A. M. Raigorodskii for the formulation of the problem and the fruitful discussion of the results. I would also like to thank M. E. Zhukovskii for a lot of helpful remarks about the text.

\end{document}